\documentclass[12pt]{amsart}%
\usepackage{amsmath}%
\usepackage{amssymb}%
\usepackage{amscd}%
\usepackage{color}%
\usepackage{enumitem}%
\usepackage{array}
\usepackage{graphicx}%
\usepackage{mathrsfs}%
\usepackage[all]{xy}%
\usepackage{stmaryrd}
\usepackage{newtxtext,newtxmath}%
\usepackage[all, warning]{onlyamsmath}
\usepackage{amsrefs}
\usepackage{bbm}
\usepackage{trimclip}
\usepackage[capitalise]{cleveref}
\usepackage[displaymath, mathlines]{lineno}
\makeatletter
\DeclareRobustCommand{\arr}{%
 \mathrel{\mathpalette\short@to\relax}%
}
\newcommand{\short@to}[2]{%
  \mkern2mu
  \clipbox{{.3\width} 0 0 0}{$\m@th#1\vphantom{+}{\shortrightarrow}$}%
  }
\makeatother
\def\wt#1{\tilde{#1}}
\theoremstyle{plain}
    
    \newtheorem{theorem}{Theorem}[section]
    \newtheorem{proposition}[theorem]{Proposition}
    \newtheorem{lemma}[theorem]{Lemma}
    \newtheorem{corollary}[theorem]{Corollary}

\theoremstyle{definition}
    \newtheorem{definition}[theorem]{Definition}
    
    \newtheorem{example}[theorem]{Example}
    \newtheorem{remark}[theorem]{Remark}
\def\Alphabet{A,B,C,D,E,F,G,H,I,J,K,L,M,N,O,P,Q,R,S,T,U,V,W,X,Y,Z}
\def\grabet{a,b,c,d,e,f,g,h,i,j,k,l,m,n,o,p,q,r,s,t,u,v,w,x,y,z}
\def\endpiece{xxx}
\def\makeAlphabet[#1]{\expandafter\makeA#1,xxx,}
\def\makealphabet[#1]{\expandafter\makea#1,xxx,}
\def\makeA#1,{\def\temp{#1}\ifx\temp\endpiece\else%
\mkbb{#1}\mkfrak{#1}\mkbf{#1}\mkcal{#1}\mkscr{#1}\mkbs{#1}\expandafter\makeA\fi}%
\def\makea#1,{\def\temp{#1}\ifx\temp\endpiece\else\mkfrak{#1}\mkbf{#1}\mkbs{#1}\expandafter\makea\fi}%
\def\mkbb#1{\expandafter\def\csname bb#1\endcsname{\mathbb{#1}}}
\def\mkfrak#1{\expandafter\def\csname fr#1\endcsname{\mathfrak{#1}}}
\def\mkbf#1{\expandafter\def\csname b#1\endcsname{\mathbf{#1}}}
\def\mkcal#1{\expandafter\def\csname c#1\endcsname{\mathcal{#1}}}
\def\mkscr#1{\expandafter\def\csname s#1\endcsname{\mathscr{#1}}}
\def\mkbs#1{\expandafter\def\csname bs#1\endcsname{{\boldsymbol{#1}}}}
\def\makeop[#1]{\xmakeop#1,xxx,}
\def\mkop#1{\expandafter\def\csname #1\endcsname{{\mathrm{#1}}}} %
\def\xmakeop#1,{\def\temp{#1}\ifx\temp\endpiece\else\mkop{#1}\expandafter\xmakeop\fi}%
\def\makeup[#1]{\xmakeup#1,xxx,}
\def\mkup#1{\expandafter\def\csname #1\endcsname{{\mathrm{#1}\,}}} %
\def\xmakeup#1,{\def\temp{#1}\ifx\temp\endpiece\else\mkup{#1}\expandafter\xmakeup\fi}%
\makeAlphabet[\Alphabet]
\makealphabet[\grabet]
\makeop[Map,Hom,alt,id,loc,col,unif,len,Supp,Consv,pr,Aut,Xxt,ab,tors,Coind,univ,sp,diam,Span,ord]
\makeup[Im,Ker,Coker,rank]
\def\e{\eta}

\def\Z{\bbZ}

\def\R{\bbR}

\def\abs#1{|#1|}

\def\Dabs#1{\bigl|\kern-0.3mm\bigl|#1\bigr|\kern-0.3mm\bigr|}
\def\C{\operatorname{Consv}^\phi}
\def\Cd{\operatorname{Consv}^{\phi'}}
\def\Ch{\operatorname{Consv}^{\hat\phi}}
\def\Cwt{\operatorname{Consv}^{\wt\phi}}

\def\vp{\varphi}

\def\s{{\boldsymbol{s}}}
\def\cp{{c_\phi}}

\def\ms{{\operatorname{ms}}}
\def\LGE{{\operatorname{LGE}}}

\def\EX{{\operatorname{ex}}}
\def\lra{\leftrightarrow}
\def\lrs{\leftrightsquigarrow}
\def\phioEX{\phi_{\text{$1$-$\EX$}}}
\def\phitwEX{\phi_{\text{$2$-$\EX$}}}
\def\phitEX{\phi_{\text{$3$-$\EX$}}}
\def\phikEX{\phi_{\text{$\kappa$-$\EX$}}}
\def\phikLGE{\phi_{\text{$\kappa$-$\LGE$}}}
\def\phitwLGE{\phi_{\text{$2$-$\LGE$}}}
\def\phitLGE{\phi_{\text{$3$-$\LGE$}}}
\def\phikmEX{\phi_{\text{$(\kappa-1)$-$\EX$}}}
\def\phiEX{\phi_\EX}

\begin{document}

\setcounter{tocdepth}{2}

\title[Interactions]{On Interactions for Large Scale Interacting Systems}
\author[Bannai]{Kenichi Bannai}\email{bannai@math.keio.ac.jp}
\author[Koriki]{Jun Koriki}
\author[Sasada]{Makiko Sasada}
\author[Wachi]{Hidetada Wachi}
\author[Yamamoto]{Shuji Yamamoto}
\thanks{This work is supported by JST CREST Grant Number JPMJCR1913 and KAKENHI 18H05233, 24K21515}
\address[Bannai, Koriki, Wachi, Yamamoto]{Department of Mathematics, Faculty of Science and Technology, Keio University, 3-14-1 Hiyoshi, Kouhoku-ku, Yokohama 223-8522, Japan.}
\address[Sasada]{Department of Mathematics, University of Tokyo, 3-8-1 Komaba, Meguro-ku, Tokyo 153-0041, Japan.}
\address[Bannai, Sasada, Wachi, Yamamoto]{Mathematical Science Team, RIKEN Center for Advanced Intelligence Project (AIP),1-4-1 Nihonbashi, Chuo-ku, Tokyo 103-0027, Japan.}

\date{\today}
\begin{abstract}
	Statistical mechanics explains the properties of macroscopic phenomena based on the movements of 
	microscopic particles such as atoms and molecules.
	Movements of microscopic particles can be represented by
	large-scale interacting systems.  In this article, we systematically study combinatorial objects which we call
	\emph{interactions}, given as symmetric directed graphs representing the possible transitions of states
	on adjacent sites of large-scale interacting systems.
	Such interactions underlie various standard stochastic processes 
	such as the \emph{exclusion processes}, 
	\emph{generalized exclusion processes},
	\emph{multi-species exclusion processes},
	\emph{lattice-gas with energy processes}, and the 
	\emph{multi-lane particle processes}.
	We introduce the notion of equivalences of interactions
	using their space of conserved quantities.
	This allows for the classification of interactions
	reflecting the expected macroscopic properties.
	In particular, we prove that when the set of local states 
	consists of \emph{two}, \emph{three} or \emph{four} elements,
	then the number of equivalence classes of separable interactions 
	are respectively \emph{one}, \emph{two} and \emph{five}.
	We also define the wedge sums and box products of interactions,
	 which give systematic methods for constructing new interactions from existing ones.
	Furthermore,  we prove that the irreducibly quantified condition for interactions,
	which implicitly plays an important role in the theory of hydrodynamic limits,
	is preserved by wedge sums and box products.
	Our results provide a systematic method to construct and classify interactions,
	offering abundant examples suitable for considering hydrodynamic limits.
\end{abstract}

\maketitle

%
%
%
\section{Introduction}\label{sec: introduction}
%
%
%

Statistical mechanics explains the properties of macroscopic phenomena based on the movements of 
microscopic particles such as atoms and molecules. For a derivation of macroscopic evolution equations from microscopic large-scale interacting particle systems, the theory of hydrodynamic limits, which relies on 
limit theorems of probability theory, provides the mathematical foundation. 
That is, if we assume that the movement of the particles are sufficiently ergodic, then 
it should be possible to obtain macroscopic behavior of the system 
via the limit theorems without knowing the exact states of each particle.
In this article, we study \emph{interactions}, which governs the movement of particles in the microscopic interacting particle systems.  We provide systematic methods to construct and classify interactions suitable for the theory of hydrodynamic limits.

In constructing the microscopic model, it is necessary to express the microscopic configuration 
of the particles as well as the possible movement of particles at an instance in time.
Rigorous mathematical framework for interpreting interacting particle systems as Markov processes 
on configuration spaces was pioneered by Spitzer \cite{Spi70} (see also \cites{Lig85,Grif93}).
One of the first examples is given by the \emph{exclusion process}, 
a process in which a site is occupied by at most one particle,
and a particle is allowed to hop to an adjacent site at a certain rate, provided that the adjacent 
site is vacant.  
This rule governing the exclusion process can be expressed rigorously as follows.
In the one dimensional case, we label the sites by the set of integers $\Z$,
and the state at one vertex $x\in\Z$ can be expressed as $0$ or $1$, where $0$ indicates that the site is vacant,
and $1$ indicates that the site is occupied by a single particle.
Then $S=\{0,1\}$ is the set of local states, that is, the possible states at each site,
and then the set of all possible configurations of the microscopic
system can be expressed as elements of the configuration space $S^\Z\coloneqq\prod_{x\in\Z}S$.

\begin{figure}[htbp]
	\begin{center}
		\includegraphics[width=0.45\linewidth]{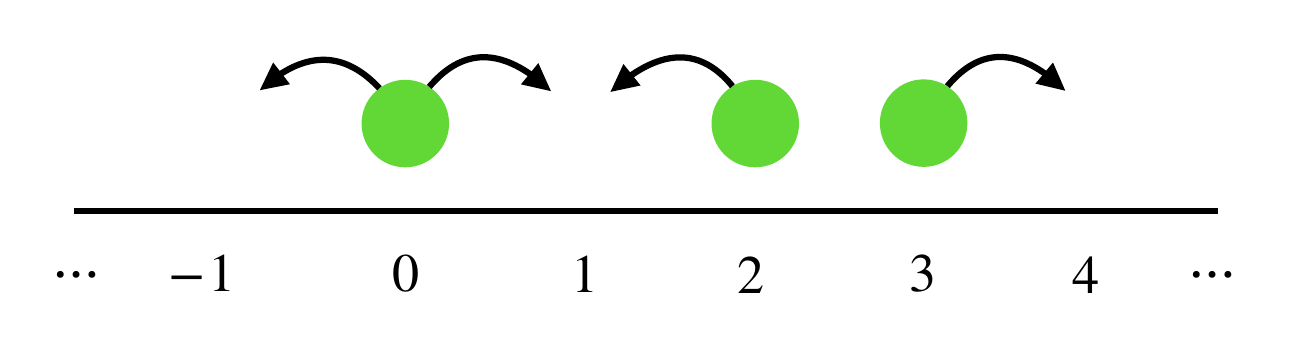}
		\caption{The Exclusion Process on $\Z$.  
		}\label{Fig: 01}
	\end{center}
\end{figure}

Our hopping rule is homogenous, in that it does not depend on the particular site.
The hopping takes place on edges $e=(x,x+1)$ between sites, for $x\in\Z$.
Hence we can express the hopping rule as the set of permitted movement between 
local states on adjacent sites.  More concretely, if we fix an edge $e=(x,x+1)$, 
then the possible configurations on sites $x$ and $x+1$ can be expressed via the set $S\times S=\{(0,0),(1,0),(0,1),(1,1)\}$.
Then $(1,0)\in S\times S$ signifies the state that $x$ is occupied by a particle and $x+1$ is vacant, 
$(1,1)\in S\times S$ signifies the state that $x$ and $x+1$ are both occupied by a particle,
etc.

\begin{figure}[htbp]
	\begin{center}
		\includegraphics[width=1\linewidth]{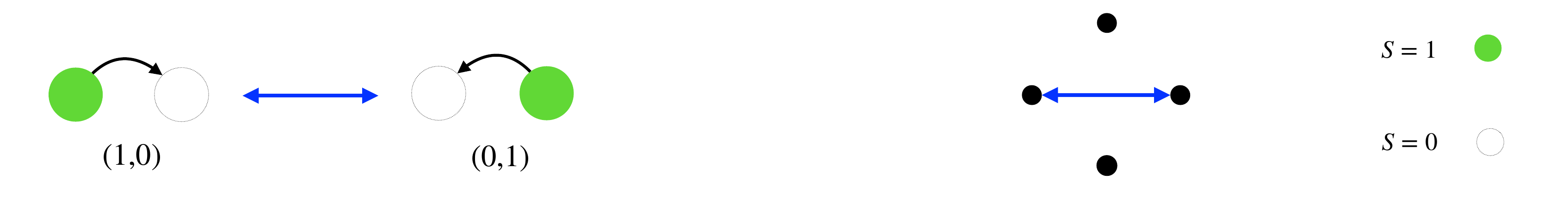}
		\caption{Exclusion Interaction and the Associated Graph}\label{Fig: 02}
	\end{center}
\end{figure}

We focus on the movement of particles on adjacent sites.
The permitted movements can be expressed as edges of a graph with vertices $S\times S$.
The hopping rule for the exclusion process is given by
\[
	\phiEX=\{ ((1,0),(0,1)), ((0,1),(1,0))  \} \subset (S\times S)\times (S\times S),
\]
which we view as the set of edges of a graph with vertices $S\times S$.
This rule is independent of our choice of edge $e$.
The pair $(S\times S,\phiEX)$ form a symmetric graph,
which is described in \cref{Fig: 02}.
We call $\phiEX$ the \emph{exclusion interaction}, and view it as
the mathematical object describing the hopping rule of the exclusion process.
In order to completely describe the stochastic property of the exclusion process, it is necessary to fix
the rate of the hopping.  This can be expressed as weights on the edges of the graph
$(S\times S,\phiEX)$.  However, for this article, we focus on the 
properties related to the underlying graph structure independent of the stochastic data.

The exclusion process is one of the most fundamentals of the interacting particle systems 
and has been studied extensively (see for example \cites{Spi70,Spo91,KL99,KLO94,KLO95,Lig99,FHU91,FUY96,GPV88,VY97}).
In generalizing the exclusion interaction in order to deal with more complicated models, 
we can consider taking the set of local states $S$ to be any non-empty set,
and consider any $\phi\subset (S\times S)\times (S\times S)$ such that the pair
$(S\times S,\phi)$ form a symmetric directed graph.  However, an arbitrary pair $(S,\phi)$ can be too general
to hope for a good theory.  One of our major interests is in determining suitable conditions on $(S,\phi)$
for the theory of hydrodynamic limit to work.
In our previous article \cite{BKS20}, we have identified the notion of \emph{irreducibly quantified interactions},
a condition which we imposed on maps $S\times S\rightarrow S\times S$.
This is a condition
implicitly satisfied by almost all known interacting particle systems with discrete $S$, and plays an important role in the proof of hydrodynamic limits.
In this article, we will extend the notion of interactions and the irreducibly quantified
condition to general symmetric directed graphs $(S\times S,\phi)$ (see \cref{def: IQ}).

Let $S$ be a non-empty set.  We define an \emph{interaction} to be 
a subset $\phi\subset (S\times S)\times(S\times S)$ such that $(S\times S,\phi)$
is a symmetric directed graph.   The exclusion above gives an example of an interaction.
We will classify the interactions using the associated space of conserved quantities.
A \emph{conserved quantity} of an interaction is a function $\xi\colon S\rightarrow\R$, such that the function
\[
	\xi_S\colon S\times S\rightarrow\R, \qquad \xi_S(s_1,s_2)\coloneqq\xi(s_1)+\xi(s_2).
\]
is constant on the connected components of the graph $(S\times S,\phi)$.
The constant function on $S$ is trivially a conserved quantity.
Philosophically, the macroscopic properties of the system should only depend on the value of the
conserved quantities and should be independent of the individual microscopic states.  
Let $\C(S)$ be the space of all conserved quantities of an interaction $(S,\phi)$ 
modulo the space of constant functions.  Then $\C(S)$
is an important invariant 
associated to the interaction.  
The dimension $\cp$ of this space corresponds to the number of independent non-constant
conserved quantities for the interaction.
We define two interactions to be \emph{equivalent},
if there exists a bijection between the sets of local states which induces
an isomorphism between the space of conserved quantities (see Definition \ref{def: equivalence} for a precise definition).
Due to its relation to the space of conserved quantities, this equivalence allows for the classification of
interactions according to the expected macroscopic properties.

Our new framework of formulating the interactions as graphs
allows for the construction of wedge sums and box products of interactions.
This gives a systematic method to construct new interactions from existing ones.
Furthermore, we prove that the irreducibly quantified condition is preserved under wedge sums and box products,
hence our method provides a systematic method to construct new irreducibly quantified interactions.
We give examples of such new interactions.
The study of hydrodynamic limits of stochastic large-scale interacting systems 
given via our interactions will be explored in subsequent research (see \cites{BS24uniform,BS24harmonic}).

The interaction is a very simple mathematical object.
Although our original interest was application to hydrodynamic limits,
the classification of interactions via conserved quantities seems to
be related to interesting computational and combinatorial problems
potentially leading to other interesting mathematical theories.  
For example, article \cite{W24} explores the problem of determining whether a given interaction satisfies the condition of being irreducibly quantified, and establishes that this problem is computationally decidable.
For this reason,
we are hoping that the study of interactions themselves can also be an interesting problem.

Although we build on many of the ideas of our previous work \cite{BKS20},
which dealt with interactions given by a map $S\times S\rightarrow S\times S$,
the current article is logically independent.  The detailed contents of this article is as follows.
In \S\ref{sec: I}, we introduce the notion of interactions and conserved quantities. Then,
we introduce the notion of equivalence of interactions, which will be the basis of the classification of interactions.
In \S\ref{sec: WS}, we introduce the wedge sums and the box products of interactions, and show that examples 
of interactions underlying typical processes such as the multi-species exclusion process,
lattice-gas with energy process and the multi-lane particle process can be obtained as 
wedge sums and box products of other simpler interactions.
In \S\ref{sec: classification}, for $\kappa\geq 1$ and
$S_\kappa=\{0,\ldots,\kappa\}$, we classify 
in \cref{thm: 1} the interactions 
whose dimension of the space of conserved quantities $\cp=\kappa-1$,
and use this fact to completely classify interactions on $S=\{0,1\}$, $S=\{0,1,2\}$ and $S=\{0,1,2,3\}$.
In \S\ref{sec: IQ}, we introduce the notion of irreducibly quantified interactions, and show that
this condition is preserved under wedge sums and box products.
Using these facts, we systematically prove in \cref{thm: main} that typical examples of interactions are 
irreducibly quantified.  In \cref{appendix: A}, we consider an interaction for $S=\{0,1,2,3\}$ 
in our classification which we believe has not previously been studied.  

%
%
%
\section{Conserved Quantities and Interactions}\label{sec: I}
%
%
%

In this section, we will consider the notion of interactions.
In this article, a \emph{graph} refers to a directed combinatorial
graph, i.e.\ a pair $(X,E)$ consisting of a set of vertices $X$ and 
the set of  edges $E\subset X\times X$.
We say that $(X,E)$ is \emph{symmetric}, if for any $e=(oe,te)\in E$, we have $\bar e=(te,oe)\in E$.
Let $S$ be a non-empty set, which we call the \emph{set of local states}.

\begin{definition}\label{def: PI}
	We define an \emph{interaction} on $S$ to be
	a subset $\phi\subset(S\times S)\times (S\times S)$
	such that the pair $(S\times S,\phi)$ form a symmetric directed graph.
	In other words,
	 if $\vp=(\s,\s')\in\phi$ for $\s,\s'\in S\times S$,
	then $\bar\vp\coloneqq(\s',\s)\in\phi$.  We will also
	often refer to the pair $(S,\phi)$ as an interaction,
	and we refer to $(S\times S,\phi)$ as the \emph{associated graph}.
\end{definition}

Note that an interaction 
in \cite{BKS20} was defined as a map
$\phi\colon S\times S\rightarrow S\times S$
satisfying a certain symmetry condition (see \cite{BKS20}*{Definition 2.4}).
The symmetrification $\phi\subset (S\times S)\times (S\times S)$
of the graph of the map $\phi$ gives an interaction in our sense of \cref{def: PI}.

We next define the conserved quantities of an interaction $\phi$ on $S$ as follows.

\begin{definition}
	For an interaction $(S,\phi)$, we define a \emph{conserved quantity}
	to be a function $\xi\colon S\rightarrow\R$ on $S$
	such that for any $\vp=(\s,\s')\in\phi$, if we let $\s=(s_1,s_2), \s'=(s'_1,s'_2)\in S\times S$,
	then
	\begin{equation}\label{eq: 1}
		\xi(s_1)+\xi(s_2)=\xi(s'_1)+\xi(s'_2).
	\end{equation}
	Note that the constant function on $S$ is trivially a conserved quantity.
\end{definition}

For an interaction $(S,\phi)$ and $\s,\s'\in S\times S$,
we write $\s\lra_\phi\s'$ or simply $\s\lra\s'$
to signify that $(\s,\s')\in\phi$, i.e. that 
the vertices $\s$ and $\s'$ are connected by an edge
in $(S\times S,\phi)$,
and we write 
$\s\lrs_\phi\s'$ or simply $\s\lrs\s'$
to signify that $\s$ and $\s'$ are in the same connected component
of $(S\times S,\phi)$.
The conserved quantity is related to the connected components of $(S\times S,\phi)$ as follows.

\begin{lemma}\label{lem: connected}
	Let $(S,\phi)$ be an interaction. 
	For any function $\xi\colon S\rightarrow\R$, 
	we let $\xi_S\colon S\times S\rightarrow\R$ be the function
	defined by 
	\begin{equation}\label{eq: xi_S}	
		\xi_S(\s)\coloneqq\xi(s_1)+\xi(s_2),\qquad\forall\s=(s_1,s_2)\in S\times S.  
	\end{equation}
	Then $\xi$ is a conserved quantity for the interaction $(S,\phi)$ if and only if the function 
	$
		\xi_S
	$
	 is constant on the connected components of the associated graph $(S\times S,\phi)$.
	 In other words, $\xi$ is a conserved quantity if and only if
	 for any $\s,\s'\in S\times S$, if $\s\lrs\s'$, then we have $\xi_S(\s)=\xi_S(\s')$.
\end{lemma}

\begin{proof}
	If $\xi_S$ is constant on the connected components of $(S\times S,\phi)$, then for any 
	$\s=(s_1,s_2)$ and $\s'=(s'_1,s'_2)$ in $S\times S$ such that
	$(\s,\s')\in\phi$,
	we see that $\xi(s_1)+\xi(s_2)=\xi_S(\s)=\xi_S(\s')=\xi(s'_1)+\xi(s'_2)$.  
	This  shows that $\xi$ is a conserved quantity.
	Conversely, suppose $\xi$ is a conserved quantity,
	and consider $\s,\s'\in S\times S$ in the same connected component of $(S\times S,\phi)$.
	Then there exists a sequence of vertices $\s^1,\ldots,\s^N\in S\times S$
	such that $\s^1=\s$, $\s^N=\s'$, and
	$(\s^i,\s^{i+1})\in\phi$ for integers $1\leq i < N$.
	Then by definition of a conserved quantity, we have
	\[
		\xi_S(\s)=\xi_S(\s^1)=\xi_S(\s^2)=\cdots=\xi_S(\s^N)=\xi_S(\s').
	\]
	This proves that $\xi_S$ is constant on the connected components of $(S\times S,\phi)$ as desired.
\end{proof}

For an interaction $(S,\phi)$,
we denote by $\C(S)$ the $\R$-linear space of conserved quantities modulo 
the space of constant functions, and let $c_{\phi}:=\dim_\R\Consv^{\phi}(S)$ be its dimension.
The exclusion interaction $\phiEX$ described in \S\ref{sec: introduction}
gives the first example of an interaction.

\begin{example}\label{def: EI}
	Let $S=\{0,1\}$.
	The \emph{exclusion interaction} $\phiEX$ is the interaction
	given by $(0,1)\lra (1,0)$.  In other words,
	\[
		\phiEX\coloneqq\{((1,0),(0,1)),((0,1),(1,0))\}\subset (S\times S)\times
		(S\times S).
	\]
	We have $c_{\phiEX}=1$, and $\Consv^{\phiEX}(S)$ is spanned by 
	$\xi\colon S\rightarrow\R$ given by $\xi(j)=j$ for $j=0,1$.
\end{example}

The generalized exclusion interaction with maximal occupancy $\kappa$, or simply the $\kappa$-exclusion
interaction, gives an example of an interaction generalizing the exclusion interaction.

\begin{figure}[htbp]
	\begin{center}
		\includegraphics[width=1\linewidth]{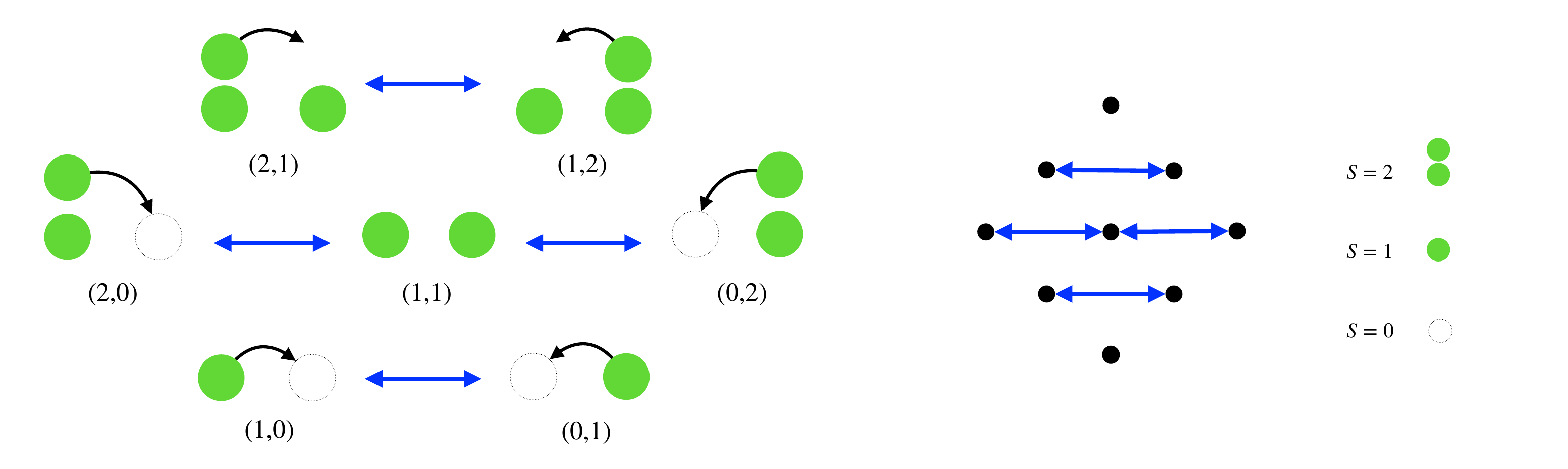}
		\caption{$\kappa$-exclusion interaction for $\kappa=2$, underlying the process
		commonly referred to as the generalized exclusion process with maximal occupancy $\kappa=2$.}\label{Fig: 03}
	\end{center}
\end{figure}

\begin{example}\label{def: GE}
	For an integer $\kappa\geq 1$, 
	let $S_\kappa=\{0,1,\ldots,\kappa\}$.
	The \emph{generalized exclusion interaction with maximal occupancy $\kappa$},
	or simply the \emph{$\kappa$-exclusion interaction}
	is the interaction  $\phikEX$ given by $ (j,k)\lra(j-1,k+1)$ for
	integers $0<j\leq\kappa$ and $0\leq k<\kappa$.
	In other words,
	\[
		\phikEX\coloneqq\{((j,k), (j-1,k+1)), 
		((j-1,k+1), (j,k))\mid 0<j\leq\kappa,\,\,0\leq k<\kappa\}.
	\]
	For any $\xi \in \Consv^{\phikEX}(S_\kappa)$ normalized as $\xi(0)=0$, by definition, $\xi(j)=\xi(j)+\xi(0)=\xi(j-1)+\xi(1)$ holds and so $\xi(j)=j \xi(1)$ for any $j=0,1,\ldots,\kappa$. Hence, we have $c_{\phikEX}=1$, and $\Consv^{\phikEX}(S_\kappa)$ is spanned by 
	$\xi\colon S_\kappa\rightarrow\R$ given by $\xi(j)=j$ for $j=0,1,\ldots,\kappa$.
	This models the movement of particles with at most $\kappa$ particles at each location.
	This interaction underlies the generalized exclusion process 
	studied by Kipnis--Landim--Olla \cites{KLO94,KLO95}, Sepp\"{a}l\"{a}inen \cite{Sep99}, Arita--Krapivsky--Mallick
	\cite{AKM14} and by many others.
	The process is called the $\kappa$-exclusion process in \cite{Sep99}.
	The case $\kappa=1$ coincides with the exclusion interaction, hence $\phioEX=\phiEX$.
\end{example}

The multi-species exclusion interaction gives another generalization of the exclusion interaction.

\begin{figure}[htbp]
	\begin{center}
		\includegraphics[width=1\linewidth]{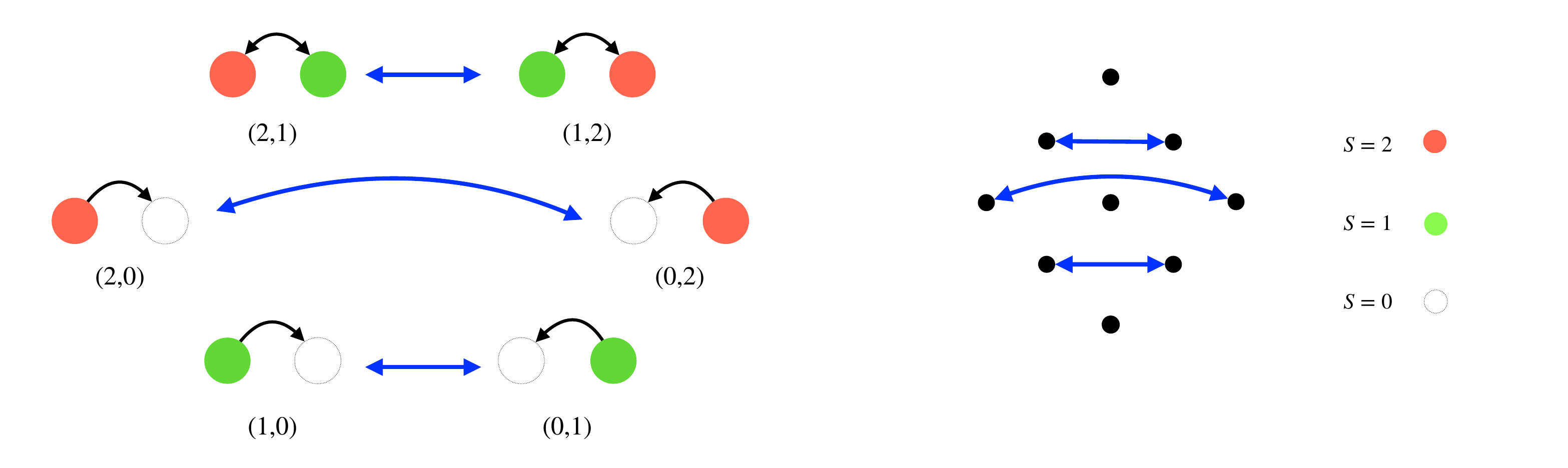}
		\caption{multi-species exclusion interaction for $\kappa=2$}
	\end{center}
\end{figure}

\begin{example}\label{def: MS}
	For an integer $\kappa\geq 1$ and $S_\kappa=\{0,1,\ldots,\kappa\}$,
	we let
	\[
		\phi^\kappa_\ms\coloneqq\{ ((j,k),(k,j))\in S_\kappa\times S_\kappa\mid j,k\in S_\kappa, j\neq k\}.
	\]
	Then the pair $(S_\kappa,\phi^\kappa_\ms)$ is an interaction, which we call the 
	\emph{multi-species exclusion interaction}.
	We have $c_{\phi^\kappa_\ms}=\kappa$, 
	and we have a basis  $\xi^1,\ldots,\xi^\kappa$ 
	of $\Consv^{\phi^\kappa_\ms}(S_\kappa)$ called the \emph{standard basis}
	defined as
	$\xi^i(0)=0$ and
	$\xi^i(j)=\delta_{ij}$ for any integer $i,j=1,\ldots,\kappa$.
	This interaction underlies the multi-color exclusion process studied by Quastel 
	\cite{Qua92} (but without the edge $(1,2) \lra (2,1)$),
	Dermoune-Heinrich \cite{DH08} and Halim--Hac\`ene \cite{HH09}, as well as
	the multi-species exclusion process studied by Nagahata--Sasada \cite{NS11}.
	Again, the case $\kappa=1$ coincides with the exclusion interaction, hence $\phi^1_\ms=\phiEX$.
\end{example}

\begin{example}\label{example: GI}
	For $S=\{-1,1\}$, the Glauber interaction $\phi_{\operatorname{G}}$ is given as in
	\cref{Fig: 05}.	The state $-1\in S$ represents a downward spin 
	and $1\in S$ represents an upward spin of a particle.
	We have $\C(S)=\{0\}$, hence $\cp=0$ in this case.
\end{example}

\begin{figure}[htbp]
	\begin{center}
		\includegraphics[width=0.5\linewidth]{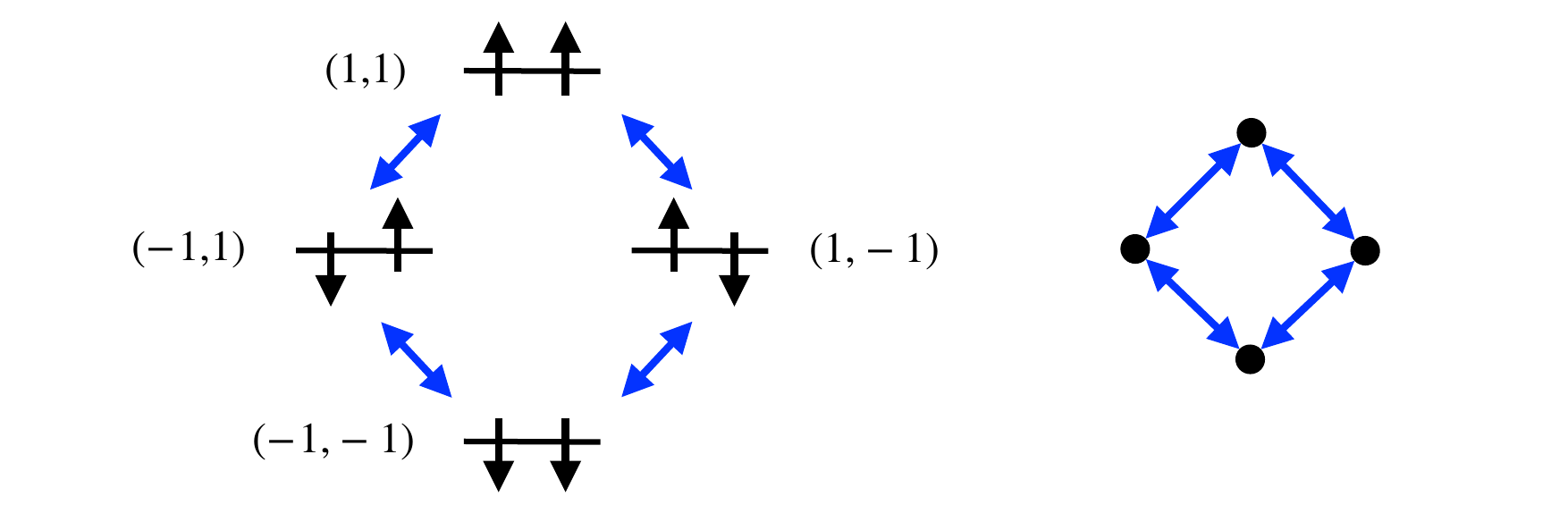}
		\caption{The Glauber interaction $\phi_{\operatorname{G}}$
		gives an example of an interaction such that $\cp=0$.}
		\label{Fig: 05}
	\end{center}
\end{figure}

Next, we define the notion of equivalence and isomorphisms of interactions.
For any set $S$, we denote by $\Map(S,\R)$ be the space of real valued functions on $S$.

\begin{definition}\label{def: equivalence}
	Let $(S,\phi)$ and $(S',\phi')$ be interactions.
	\begin{enumerate}
	\item
	We say that $(S,\phi)$ and $(S',\phi')$ are \emph{equivalent},
	if there exists a bijection $S\cong S'$ such that the induced isomorphism
	$\Map(S',\R)\cong\Map(S,\R)$ induces an $\R$-linear isomorphism 
	$\Cd(S')\cong\C(S)$.  In this case, we denote 
	$(S,\phi)\simeq(S',\phi')$, or simply, $\phi\simeq\phi'$.
	\item
	We say that $(S,\phi)$ and $(S',\phi')$ are \emph{isomorphic},
	if there exists an equivalence $(S,\phi)\simeq(S',\phi')$ mapping $\phi$ bijectively to $\phi'$
	that gives an isomorphism of  graphs $(S\times S,\phi)\cong (S'\times S',\phi')$.
	In this case, we denote 
	$(S,\phi)\cong(S',\phi')$, or simply, $\phi\cong\phi'$.
	\end{enumerate}
\end{definition}

Both equivalence and isomorphisms are equivalence relations of interactions.
The following example is isomorphic to the $2$-exclusion interaction.

\begin{figure}[htbp]
	\begin{center}
		\includegraphics[width=1\linewidth]{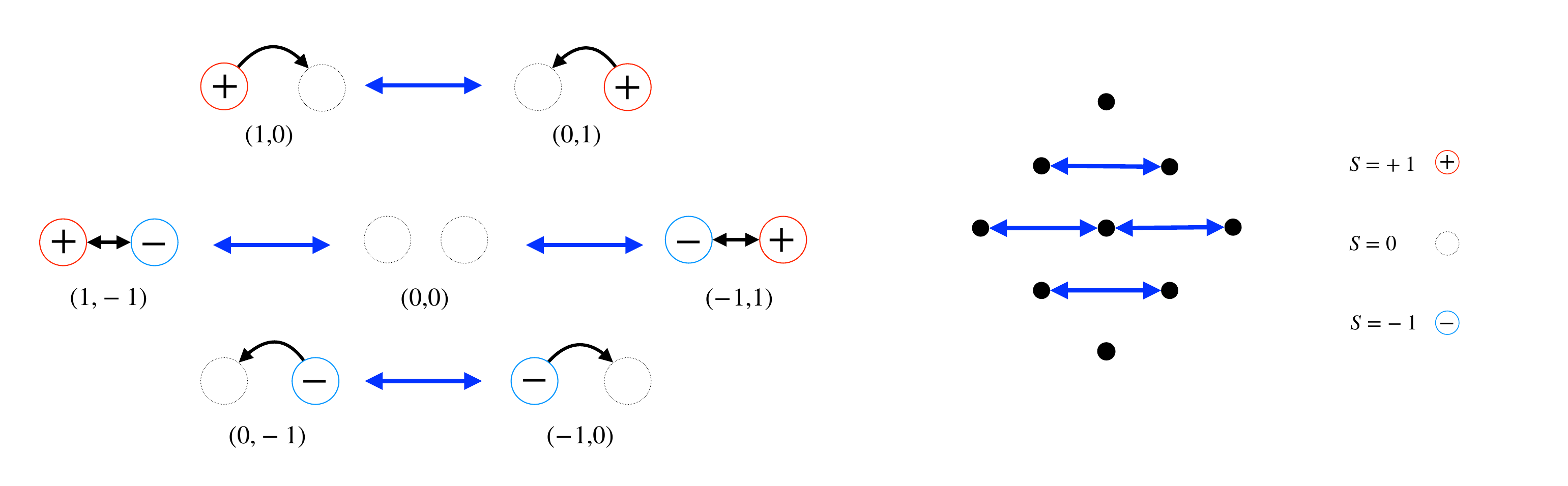}
		\caption{Interaction of \cref{example: 1} underlying the two-species exclusion 
	process with annihilation and creation studied in \cite{Sas10}.  
	Note that the associated graph coincides with that of \cref{Fig: 03}.}
	\end{center}
\end{figure}

\begin{example}\label{example: 1}
	Let $S=\{-1,0,+1\}$, and we let $\phi\subset(S\times S)\times(S\times S)$ 
	be the interaction given by
	\begin{align*}
		(-1,0)&\lra (0,-1), & (1,0)&\lra (0,1), &
		(1,-1)&\lra(0,0)\lra(-1,1).
	\end{align*}
	Then $(S,\phi)$ is an interaction underlying the two-species exclusion 
	process with annihilation and creation studied by Sasada in \cite{Sas10}.	
	Let $S'\coloneqq\{0,1,2\}$ and consider the bijection
	$S'\cong S$ given by $0\mapsto -1, 1\mapsto 0, 2\mapsto 1$.
	Then the interaction $\phi'$ on $S'$ corresponding to $(S,\phi)$ is 
	given by
	\begin{align*}
		(0,1)&\lra (1,0), & (2,1)&\lra (1,2), &
		(2,0)&\lra(1,1)\lra(0,2).
	\end{align*}
	Then we see that $(S',\phi')$ coincides with the $2$-exclusion interaction
	$\phitwEX$ on $\{0,1,2\}$.
	Thus we have $c_{\phi}=1$, and $\Consv^\phi(S)$ is spanned by 
	$\xi\colon S\rightarrow\R$ given by $\xi(j)=j$ for $j=-1,0,1$.
\end{example}

\cref{example: 1} shows that even if the interactions are isomorphic, it can underly seemingly very different stochastic models.
However, the proof of hydrodynamic limit in \cite{Sas10} was possible
specifically because the underlying geometry of the model coincided with the case of the $2$-exclusion interaction.
This is one of the reasons that we are interested in the classification of interactions via isomorphisms and equivalences.

Given an interaction $(S,\phi)$, we can construct an equivalent interaction $\hat\phi$ on $S$
whose connected components are classified by the values of the conserved quantities.

\begin{lemma}\label{lem: completion}
	For any interaction $(S,\phi)$, we define the completion $\hat\phi$ on $S$ by
	\[
		\hat\phi\coloneqq\{ (\s,\s')\in (S\times S)\times(S\times S)\mid \xi_S(\s)=\xi_S(\s')\,\,\forall\xi\in\C(S)\}.
	\]
	Then 
	$\phi\subset\hat\phi$, and $\hat\phi$ is equivalent to $\phi$.
	The connected components of the graph $(S\times S,\hat\phi)$ are classified
	by the values of the conserved quantities.
\end{lemma}

\begin{proof}
	Our first assertion follows from the construction
	and \cref{lem: connected}.
	By construction and \cref{lem: connected}, for any $\s,\s'\in S\times S$,
	we have
	$\s\lra\s'$ if and only if $\xi_S(\s)=\xi_S(\s')$
	for any $\xi\in\operatorname{Consv}^{\hat\phi}(S)=\C(S)$.
	This gives our second assertion.
\end{proof}

If we assume that $S$ is finite, then any interaction $(S,\phi)$ is isomorphic to
an interaction on $\{0,\ldots,\kappa\}$, where $\kappa+1=\abs{S}$ 
is the number of elements in $S$.
Due to the following \cref{prop: smallest},  we see that 
any interaction on a finite set of local states is 
equivalent to an interaction obtained by adding edges to the associated graph $(S\times S,\phi^\kappa_\ms)$ of the
multi-species exclusion interaction.

\begin{proposition}\label{prop: smallest}
	Suppose $\phi$ is an interaction on $S=\{0,\ldots,\kappa\}$.  
	Then for the associated completion
	$\hat\phi$ on $S$, we have
	$\phi^\kappa_\ms\subset\hat\phi$ for the multi-species exclusion interaction $\phi^\kappa_\ms$ on $S$.
	In particular, any interaction $\phi$ on $S$ such that $\C(S)=\kappa$ is equivalent to the multi-species exclusion interaction.
\end{proposition}	

\begin{proof}
	Consider any $((j,k),(k,j))\in\phi^\kappa_\ms$, where $j,k\in S$ such that $j\neq k$
	from the definition of the multi-species exclusion interaction of \cref{def: MS}.
	For any $\xi\in\C(S)$, we have
	\[
		\xi_S(j,k)=\xi(j)+\xi(k)=\xi(k)+\xi(j)=\xi_S(k,j).
	\]	
	Hence if we let $\hat\phi$ be the completion of $\phi$
	constructed in \cref{lem: completion}, 
	then we have $((j,k),(k,j))\in\hat\phi$.  
	This shows that $\phi^\kappa_\ms\subset\hat\phi$.
	If $\C(S)=\kappa$, then since $\Consv^{\phi^\kappa_\ms}(S)=\kappa$, 
	the natural inclusion $\C(S)=\Ch(S)\subset \Consv^{\phi^\kappa_\ms}(S)$ 
	induces an isomorphism
	\[
		\C(S)\cong\Consv^{\phi^\kappa_\ms}(S),
	\]
	hence by the definition of equivalence given in \cref{def: equivalence}, 
	we see that the interaction $\phi$ is equivalent to $\phi^\kappa_\ms$ as desired.
\end{proof}

%
%
%
\section{Wedge Sums and Box Products of Interactions}\label{sec: WS}
%
%
%

In this section, we introduce the wedge sums and box products of interactions, which are methods to
construct new interactions from existing ones.
In fact, the multi-species exclusion interaction can be obtained as iterated wedge sums of the exclusion interaction.
We will use such constructions to give various examples of interactions,
such as the lattice-gas with energy and the multi-lane interactions.

Let $S_1$ and $S_2$ be non-empty sets.  For $*_1\in S_1$ and $*_2\in S_2$,
we define the \emph{wedge sum} of $S_1$ and $S_2$ along $*_1\in S_1$ and $*_2\in S_2$ by
\[
	S_1\vee S_2\coloneqq (S_1\amalg S_2)/\sim,
\]
where $\sim$ is the equivalence relation on $S_1\amalg S_2$ generated by $*_1\sim *_2$.
We denote by $*$ the element represented by $*_1$ and $*_2$ in $S_1\vee S_2$.

\begin{definition}\label{def: WS}
	Let $\phi_1$ and $\phi_2$ be interactions on $S_1$ and $S_2$,
	and let $*_1\in S_1$ and $*_2\in S_2$.
	We define the interaction $\phi_1\vee\phi_2$ on $S_1\vee S_2$
	to be the interaction given by
	\begin{align*}
		\bss&\lra\bss', \qquad   \text{$(\bss,\bss')\in \phi_1$ or $(\bss,\bss')\in\phi_2$}\\
		(s,s')&\lra (s',s),\qquad \text{$s\in S_1$ and $s'\in S_2$}.
	\end{align*}
	We call the interaction $(S_1\vee S_2,\phi_1\vee\phi_2)$
	the \emph{wedge sum} of $(S_1,\phi_1)$ and $(S_2,\phi_2)$ 
	along $*_1\in S_1$ and $*_2\in S_2$, and 
	we denote this interaction as $(S_1,\phi_1)\vee(S_2,\phi_2)$.
\end{definition}

\begin{proposition}\label{prop: WS}
	Let $(S_1,\phi_1)$ and $(S_2,\phi_2)$ be interactions.  For $*_1\in S_1$ and $*_2\in S_2$,
	consider the wedge sum $S_1\vee S_2$ along $*_1$ and $*_2$.
 	Then the wedge sum $(S_1,\phi_1)\vee(S_2,\phi_2)=(S_1\vee S_2,\phi_1\vee\phi_2)$
	satisfies 
	\[
		\Consv^{\phi_1\vee\phi_2}(S_1\vee S_2)=\Consv^{\phi_1}(S_1)\oplus\Consv^{\phi_2}(S_2).
	\] 
	In particular, we have $c_{\phi_1\vee\phi_2}=c_{\phi_1}+c_{\phi_2}$.
\end{proposition}

\begin{proof}
	For any conserved quantity $\xi_1\in\Consv^{\phi_1}(S_1)$,
	take a representative so that $\xi_1(*_1)=0$.
	Consider the function $\xi$ on $S\coloneqq S_1\vee S_2$ given by
	$\xi(s)\coloneqq\xi_1(s)$ if $s\in S_1$ and $\xi(s)\coloneqq 0$ if $s\in S_2$.
	Note that $\xi$ is well-defined since $\xi(*)=\xi_1(*_1)=0$.
	Then for $\bss,\s'\in S_1\times S_1\subset S\times S$ 
	such that $(\s,\s')\in\phi_1$, we have $\xi_{S}(\s)=(\xi_1)_{S_1}(\s)=(\xi_1)_{S_1}(\s')=\xi_S(\s')$,
	and for $\s,\s'\in S_2\times S_2$, we have $\xi_S(\s)=0=\xi_S(\s')$.   Moreover, $\xi_S(s_1,s_2)=
	\xi_1(s_1)=\xi_S(s_2,s_1)$ for any $s_1\in S_1$ and $s_2\in S_2$.
	This shows that $\xi\in\C(S)$, and the correspondence $\xi_1\mapsto\xi$ gives an injection
	$\Consv^{\phi_1}(S_1)\hookrightarrow\C(S)$.
	Similar construction for $\xi_2\in\Consv^{\phi_2}(S_2)$ gives an injection
	\begin{equation}\label{eq: injection}
		\Consv^{\phi_1}(S_1)\oplus\Consv^{\phi_2}(S_2)\hookrightarrow\C(S).
	\end{equation}
	Next, consider any conserved quantity  $\xi\in\C(S)$  normalized so that $\xi(*)=0$.
	The restriction of $\xi$ to $S_1$ and $S_2$ give
	conserved quantities $\xi_1\in\Consv^{\phi_1}(S_1)$ and $\xi_2\in\Consv^{\phi_2}(S_2)$
	such that $\xi=\xi_1+\xi_2$ with respect to the inclusion \eqref{eq: injection}.
	This proves that \eqref{eq: injection} is an isomorphism as desired.
\end{proof}

\begin{remark}
	For any integer $\kappa\geq 1$ and $S_\kappa=\{0,1,\ldots,\kappa\}$, let $(S_\kappa,\phi^\kappa_\ms)$
	be the multi-species exclusion interaction.  
	For integers $\kappa,\kappa'\geq 1$, consider the
	wedge sum $S_{\kappa}\vee S_{\kappa'}$ along
	$0\in S_\kappa$ and $0\in S_{\kappa'}$.
	Then the bijection $S_{\kappa}\vee S_{\kappa'}\cong S_{\kappa+\kappa'}$
	given by mapping
	$j\in S_\kappa$ to $j\in S_{\kappa+\kappa'}$ and
	 $j\in S_{\kappa'}$ such that $j\geq 1$
	to $j+\kappa\in S_{\kappa+\kappa'}$
	induces an isomorphism of interactions
	\[
		 \phi^{\kappa+\kappa'}_\ms
		\cong\phi_\ms^\kappa\vee\phi^{\kappa'}_\ms.
	\]
	In particular, noting that $(S_1,\phi^1_\ms)$ coincides with the exclusion interaction $(S_1,\phiEX)$, we have
	\[
		\phi^{\kappa}_\ms\cong\vee^\kappa\phiEX\coloneqq\phiEX\vee\cdots\vee\phiEX
	\]
	for any integer $\kappa\geq 1$, where the right hand side is the $\kappa$-fold
	wedge sum along $0\in S_1$ of the exclusion interaction $(S_1,\phiEX)$.
	In this case, 
	the bijection $S_1\vee\cdots\vee S_1\cong S_\kappa$ is given by mapping
	$1\in S_1$ in the $j$-th component of $S_1\vee\cdots\vee S_1$ 
	to $j\in S_\kappa$.
\end{remark}

\begin{figure}[htbp]
	\begin{center}
		\includegraphics[width=0.8\linewidth]{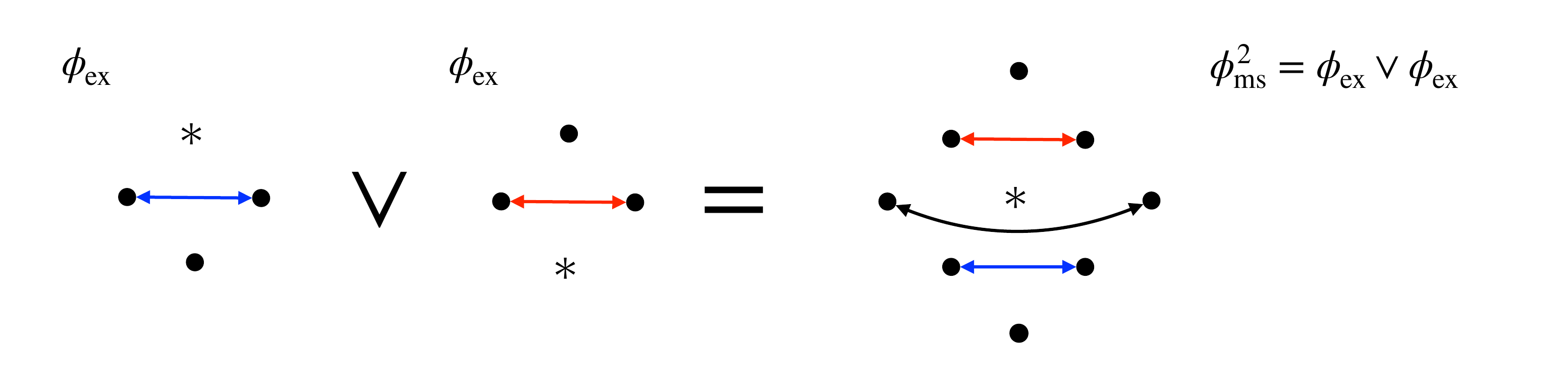}
		\caption{The wedge sum $\phiEX\vee\phiEX$ of two exclusion interactions 
		along $*$ gives the multi-species interaction $\phi^2_\ms$. The blue arrow indicates the
		edge coming from the first exclusion, the red arrow indicates the
		edge coming from the second exclusion and the black arrow indicates
		the new edge added via definition of the wedge sum.}
		\label{Fig: 7}
	\end{center}
\end{figure}

We next consider the lattice gas with energy interaction, which is obtained as the wedge sum
of the exclusion interaction and the $\kappa$-exclusion interaction.
This interaction underlies the process studied by Nagahata in \cite{N03}.

\begin{example}\label{def: LGE}
	For any integer $\kappa\geq 1$, let $S_\kappa=\{0,1,\ldots,\kappa\}$.
	We let $S_1\vee S_{\kappa-1}$ be the wedge sum of $S_1$ and $S_{\kappa-1}$ along $1\in S_1$ and $0\in S_{\kappa-1}$.
	We define the \emph{lattice gas with energy}
	to be the interaction
	\[
		\phikLGE\coloneqq\phiEX\vee\phikmEX
	\]
	on $S_1\vee S_{\kappa-1}$.
\end{example}

The lattice gas with energy has the following interpretation.

\begin{lemma}\label{lem: LGE}
	For $\kappa\geq 1$, let $\phikLGE$ be the lattice gas with energy interaction
	defined in \cref{def: LGE}.
	Then $(S_1\vee S_{\kappa-1},\phikLGE)$ is isomorphic to the interaction 
	$(S_{\kappa},\phi)$ for $S_{\kappa}=\{0,1,\ldots,\kappa\}$
	and $\phi\subset (S_{\kappa}\times S_{\kappa})\times (S_{\kappa}\times S_{\kappa})$
	given by 
	\begin{align*}
			(j, k)&\lra(j-1,k+1) \qquad \text{$j\geq 2$ and $k< \kappa$},\\
			(j,k)&\lra(k,j)  \hspace{2.1cm} j=0 \ \text{or}  \ 1, k \in S_{\kappa}, j\neq k 
	\end{align*}
	Moreover, we have $\cp=2$, and
	 $\C(S_{\kappa})$ is spanned by
	$\xi_P$ and $\xi_E$,	
	where
	\begin{align*}
		\xi_P(j)&\coloneqq 
		\begin{cases}
			1 & j\geq 1\\
			0 &\text{otherwise},
		\end{cases}
		&
		\xi_E(j)\coloneqq j
	\end{align*}
	for any $j\in S_\kappa$.
	In particular,  $\phikLGE$ 
	models the movement of particles with maximal energy $\kappa$.
	Then the conserved quantity $\xi_P$ gives the number of particles,
	and $\xi_E$ gives the energy.
\end{lemma}

\begin{proof}
	We have a bijection $S_1\vee S_{\kappa-1}\cong S _{\kappa}$, 
	given by mapping $0,1\in S_1$ respectively to $0,1\in S_{\kappa}$,
	and $j\in S_{\kappa-1}$ to $j+1\in S_{\kappa}$.
	Then $\phikLGE$ corresponds to $\phi$ through this identification.
	The fact that $\cp=2$ follows from the isomorphism
	\[
		\Consv^{\phikLGE}(S_{\kappa})\cong\Consv^{\phiEX}(S_1)\oplus\Consv^{\phikmEX}(S_{\kappa-1})
	\]
	given in  \cref{prop: WS}.
	The conserved quantity $\xi\in\Consv^{\phiEX}(S_1)$
	given by $\xi(0)=-1$ and $\xi(1)=0$ maps to the conserved quantity
	$\xi'$ on $S_\kappa$ such that $\xi'(0)=-1$ and $\xi'(j)=0$ for $j\geq 1$,
	which by addition of $1$ is equivalent to
	$\xi_P\in\Consv^{\phikLGE}(S_{\kappa})$.
	Furthermore,
	the conserved quantity $\xi\in\Consv^{\phikmEX}(S_{\kappa-1})$
	given by $\xi(j)=j$ for $j\in S_{\kappa-1}$ maps to a conserved quantity
	 equivalent via addition by $1$ to 
	$\xi_E\in\Consv^{\phikLGE}(S_{\kappa})$.
\end{proof}

\begin{remark}
	Due to \cref{lem: LGE}, we will often denote 
	the lattice gas with energy interaction
	$(S_1\vee S_{\kappa-1},\phikLGE)$ as $(S_\kappa,\phikLGE)$.
	The difference between $\phikLGE$ and $\phikEX$ on $S_\kappa$
	is that 
	\begin{align*}
		(1,k)&\lra_{\phikEX} (0,k+1),&
		(1,k)&\nleftrightarrow_{\phikLGE} (0,k+1) 
	\end{align*}
	 for any $0<k< \kappa$.  Note that we have an isomorphism $(S_2,\phitwLGE)\cong (S_2,\phi^2_\ms)$.
	 In other words, lattice gas with energy interaction coincides with the multi-species exclusion interaction
	 if $\kappa=2$.
\end{remark}

\begin{figure}[htbp]
	\begin{center}
		\includegraphics[width=0.8\linewidth]{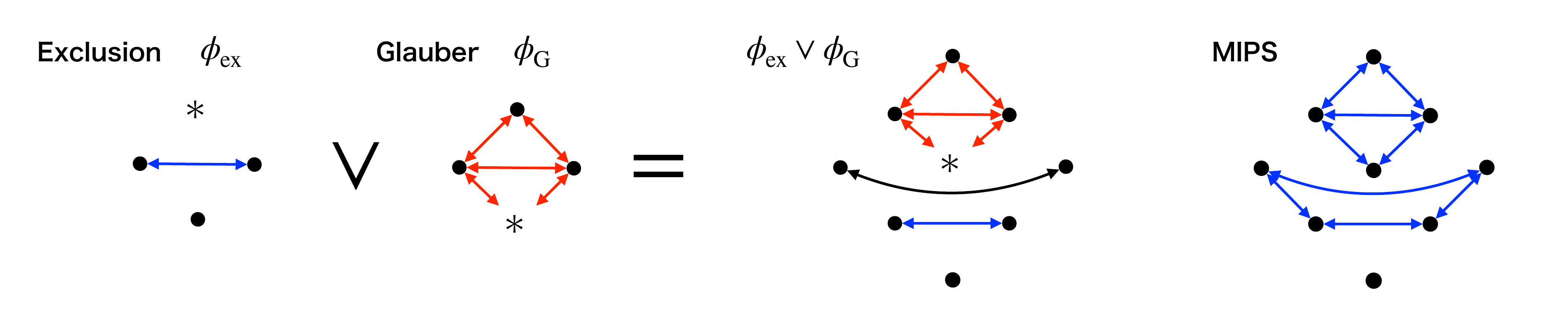}
		\caption{
		The wedge sum of the exclusion interaction $\phiEX$
		and the Glauber interaction $\phi_{\operatorname{G}}$
		is easily checked to be equivalent to the MIPS interaction described in the figure
		underlying the MIPS process studied in \cites{ATK22,KEBT18}.}
		\label{Fig: 8}
	\end{center}
\end{figure}

We next define the notion of box products of interactions to construct new interactions.

\begin{definition}\label{def: product}
	Let $(S_1,\phi_1)$ and $(S_2,\phi_2)$ be interactions.  We let 
	$S_1\times S_2$ be the product of sets, and we define the \emph{box product}
	$\phi_1\square \phi_2$ to be the interaction on $S_1\times S_2$
	given by
	\[
		\phi_1\square\phi_2\coloneqq\{((\s_1,\s_2),(\s'_1,\s'_2)) \mid 
		\text{$\s_1=\s'_1$ and $(\s_2,\s'_2)\in\phi_2$ or $\s_2=\s'_2$ and $(\s_1,\s'_1)\in\phi_1$}\}.
	\]
	We denote $(S_1\times S_2,\phi_1\square\phi_2)=(S_1,\phi_1)\square (S_2,\phi_2)$.
\end{definition}

The box product is well-behaved for exchangeable interactions, defined as follows.

\begin{definition}\label{def: exchangeable}
	We say that an interaction $(S,\phi)$ is \emph{exchangeable},
	if $(s_1,s_2)\lrs(s_2,s_1)$ for any $(s_1,s_2)\in S\times S$.
\end{definition}

The $\kappa$-exclusion and the multi-species interactions give examples of exchangeable interactions.
Moreover, we have the following:

\begin{lemma}
	Any interaction is equivalent to an interaction which is exchangeable.
\end{lemma}

\begin{proof}
	For any interaction $(S,\phi)$, consider the completion $(S,\hat\phi)$ given in \cref{lem: completion}.
	Then for any $(s_1,s_2)\in S\times S$, we have $\xi_S(s_1,s_2)=
	\xi(s_1)+\xi(s_2)=
	\xi_S(s_2,s_1)$ for any $\xi\in\C(S)$.
	Then by the definition of the completion, we have $(s_1,s_2)\lra_{\hat\phi}(s_2,s_1)$.
	This shows that $(S,\hat\phi)$ is exchangeable.
	Our assertion follows from the fact that $(S,\hat\phi)$ is equivalent to $(S,\phi)$.
\end{proof}

We prove that a box product of exchangeable interactions
is an interaction such that the space of conserved quantities 
is the direct sum of the space of conserved quantities of each of the components.

\begin{figure}[htbp]
	\begin{center}
		\includegraphics[width=1\linewidth]{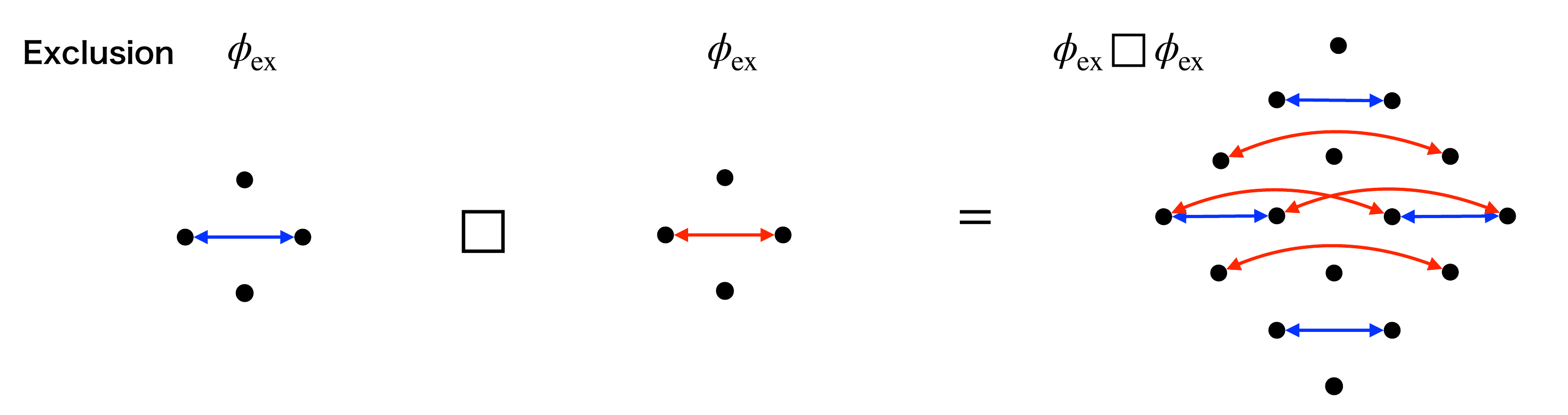}
		\caption{The figure represents the box product $\phiEX\square\phiEX$.  The blue arrow indicates the
		edges coming from the first exclusion and the red arrow indicates the edges coming from the
		second exclusion.}\label{{Fig: 07}}
	\end{center}
\end{figure}

\begin{proposition}\label{prop: dimension} 
	Let $(S_1,\phi_1)$ and $(S_2,\phi_2)$ be interactions.
	 If either $\phi_1$ or $\phi_2$ is exchangeable, then
	 we have a canonical isomorphism
		\[
			\Consv^{\phi_1\square\phi_2}(S_1\times S_2)
			\cong\Consv^{\phi_1}(S_1)\oplus \Consv^{\phi_2}(S_2)
		\]
	for the box product 
	$(S_1\times S_2,\phi_1\square\phi_2)=(S_1,\phi_1)\square(S_2,\phi_2)$ 
	of \cref{def: product}.
	In particular, we have $c_{\phi_1\square\phi_2}=c_{\phi_1}+c_{\phi_2}$.
	Furthermore, if 	both $\phi_1$ and $\phi_2$ are exchangeable, then 
	the box product
	$\phi_1\square\phi_2$ is also exchangeable.
\end{proposition}

\begin{proof}
	For any $\xi_1\in\Consv^{\phi_1}(S_1)$ and $\xi_2\in\Consv^{\phi_2}(S_2)$,
	from the definition of the box product
	$\phi_1\square\phi_2$, the sum
	$\xi\coloneqq\xi_1+\xi_2\colon S_1\times S_2\rightarrow\R$ defined by
	$\xi(s_1,s_2)\coloneqq\xi_1(s_1)+\xi_2(s_2)$ for $(s_1,s_2)\in S_1\times S_2$
	gives a conserved quantity for $\phi_1\square\phi_2$ on $S\coloneqq S_1\times S_2$.
	This gives a linear injection
	\begin{equation}\label{eq: inclusion}
			\Consv^{\phi_1}(S_1)\oplus \Consv^{\phi_2}(S_2)
			\hookrightarrow\Consv^{\phi_1\square\phi_2}(S).
	\end{equation}
	Next, assume without loss of generality that $\phi_2$ is exchangeable.
	In order to prove that \cref{eq: inclusion} is surjective,
	consider any $\xi\in\Consv^{\phi_1\square\phi_2}(S)$.
	We fix $*_1\in S_1$ and $*_2\in S_2$, and by replacing $\xi$ with $\xi-\xi(*_1,*_2)$,
	we can assume that we have $\xi(*_1,*_2)=0$.
	Let $\xi_1(s_1)\coloneqq\xi(s_1,*_2)$ and $\xi_2(s_2)\coloneqq\xi(*_1,s_2)$
	for any $(s_1,s_2)\in S_1\times S_2$.  Since $\xi$ is a conserved quantity for 
	$\phi_1\square\phi_2$, the definition of the box product implies that $\xi_1$ and
	$\xi_2$ are conserved quantities for $\phi_1$ and $\phi_2$.
	Since we have assumed that $\phi_2$ is exchangeable, 
	we have $(s_2,*_2)\lrs_{\phi_2}(*_2,s_2)$ for any $s_2\in S_2$.
	Then the definition of the box product gives 
	$((s_1,s_2),(*_1,*_2))\lrs_{\phi_1\square\phi_2}((s_1,*_2),(*_1,s_2))$.
	Since $\xi$ is a conserved quantity for $\phi_1\square\phi_2$, this shows that
	\[
		\xi(s_1,s_2)=\xi(s_1,s_2)+\xi(*_1,*_2)=\xi(s_1,*_2)+\xi(*_1,s_2)=\xi_1(s_1)+\xi_2(s_2)
	\]
	for any $(s_1,s_2)\in S_1\times S_2$.  This implies that $\xi=\xi_1+\xi_2$,
	hence that $\xi$ is in the image of the map \cref{eq: inclusion}.
	This proves that \cref{eq: inclusion} is an isomorphism as desired.
	
	Next, assume that $\phi_1$ and $\phi_2$ are exchangeable,
	and consider $((s_1,s_2),(s'_1,s'_2))\in S\times S$.
	Since $\phi_1$ and $\phi_2$ are exchangeable, 
	we have $(s_1,s'_1)\lrs_{\phi_1}(s'_1,s_1)$
	and $(s_2,s'_2)\lrs_{\phi_2}(s'_2,s_2)$. 
	 By definition
	of the box product $\phi_1\square\phi_2$, this implies that 
	$((s_1,s_2),(s'_1,s'_2))\lrs_{\phi_1\square\phi_2}
	((s'_1,s_2),(s_1,s'_2))$ and $((s'_1,s_2),(s_1,s'_2))\lrs_{\phi_1\square\phi_2}
	((s'_1,s'_2),(s_1,s_2))$,
	hence $((s_1,s_2),(s'_1,s'_2))\lrs_{\phi_1\square\phi_2}((s'_1,s'_2),(s_1,s_2))$.
	This proves that $(S,\phi_1\square\phi_2)$ is exchangeable as desired.
\end{proof}

Using the box product, we can define
the $N$-lane particle interactions.

\begin{example}\label{def: RP}
	For any integer $\kappa\geq 1$ and $S_\kappa=\{0,1,\ldots,\kappa\}$, let 
	$(S_\kappa,\phikEX)$ be the $\kappa$-exclusion interaction of \cref{def: GE}.
	For an integer $N\geq 1$,
	we define the \emph{$N$-lane $\kappa$-exclusion interaction} 
	to be the interaction
	\[
		\phi^{\square N}_{\text{$\kappa$-$\EX$}}\coloneqq\phikEX\square\cdots\square\phikEX,
	\]
	where the right hand side of the above equality is the $N$-fold box product
	of the $\kappa$-exclusion interaction $\phikEX$.
	In particular, we have $\phi^{\square 1}_{\text{$\kappa$-$\EX$}}=\phikEX$.
\end{example}

For the $N$-lane $\kappa$-exclusion interaction,
the state $(m_1,\ldots,m_N)\in \{0,1,\ldots,\kappa\}^N$ signifies the state with $m_i$ particles in lane $i$ for $i=1,\ldots, N$.

\begin{figure}[htbp]
	\begin{center}
		\includegraphics[width=1\linewidth]{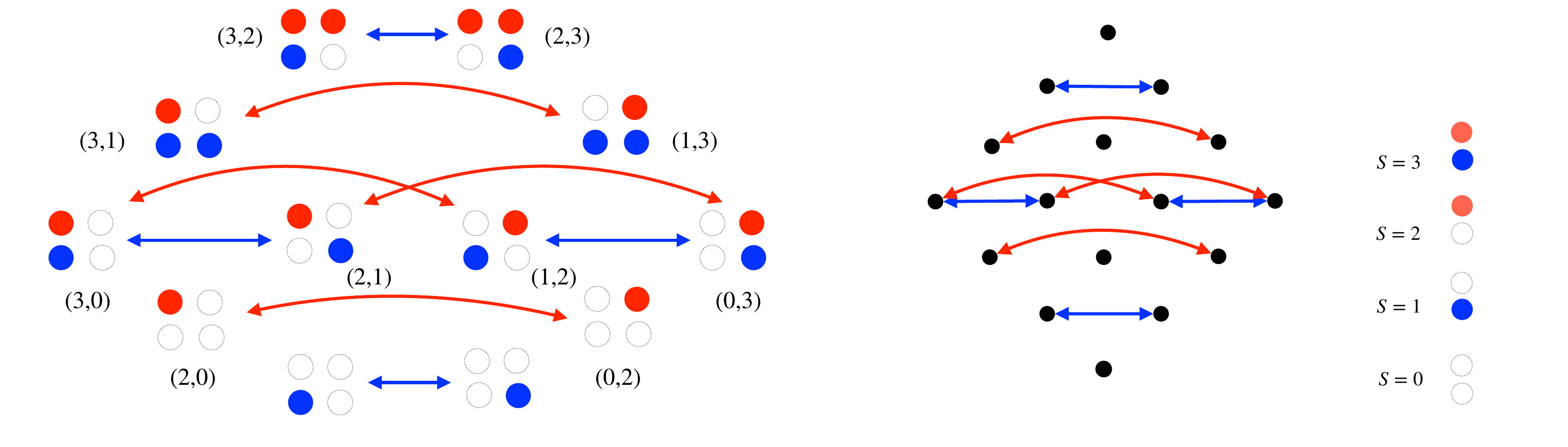}
		\caption{The figure describes the two-lane $1$-exclusion interaction $\phi^{\square 2}_\EX=\phiEX\square\phiEX$.
		The blue particles are in the first lane and the red particles are in the second lane.
		The particles move in their respective lanes.
		The movements in each lane coincides with the exclusion interaction.}\label{{Fig: 08}}
	\end{center}
\end{figure}

\begin{lemma}\label{lem: RL}
	The $N$-lane particle  interaction
	$\phi^{\square N}_{\text{$\kappa$-$\EX$}}\subset S_\kappa^N\times S_\kappa^N$ for $S_\kappa=\{0,1,\ldots,\kappa\}$ is given by 
	\begin{align*}
		((m_1,\ldots,m_i,\ldots,m_N),& (n_1,\ldots,n_i,\ldots,n_N))\\
		&\lra
		((m_1,\ldots,m_i-1,\ldots,m_N), (n_1,\ldots,n_i+1,\ldots,n_N))
	\end{align*}
	for $(m_1,\ldots,m_N)$ and $(n_1,\ldots,n_N)$ in $S_\kappa^N$ such that $m_i\geq 1$ and $n_i<\kappa$.
	We have $\cp=N$, and $\Consv^{\phi^{\square N}_{\text{$\kappa$-$\EX$}}}(S_\kappa^N)$ is spanned by
	$\xi^i$ for $i=1,\ldots,N$,
	where 
	\begin{align*}
		\xi^i((m_1,\ldots,m_N))&\coloneqq m_i.
	\end{align*}
	The conserved quantity $\xi^i$ gives the number of particles in lane $i$ for $i=1,\ldots,N$.
\end{lemma}

\begin{proof}
	This follows from the definition of the boxed product and the $N$-lane $\kappa$-exclusion interaction.
\end{proof}

%
%
\section{Classification of Interactions}\label{sec: classification}
%
%
%

In this section, we give methods to classify interactions for finite $S$.
In what follows, assume that $S$ is finite.
Then $\Consv^{\phi}(S)$ is of finite dimension
for any interaction $(S,\phi)$.
We first prove that
by connecting two distinct connected components of $(S\times S,\phi)$, we can construct a new 
interaction $\wt\phi$ on $S$ containing $\phi$ whose dimension of the space of conserved quantities 
$\dim_\R\Cwt(S)$ is
\emph{one less} than that of $\phi$.

\begin{lemma}\label{prop: less}
	Let $(S,\phi)$ be an interaction.
	For $\s,\s'\in S\times S$, suppose there exists $\xi\in\Consv^{\phi}(S)$ such that $\xi_S(\s)\neq\xi_S(\s')$.
	Then there exists an interaction $\wt\phi$ on $S$ unique up to equivalence such that
	$\phi\subset\wt\phi$, 
	\begin{equation}\label{eq: dim}
		\dim_\R\Cwt(S)=\dim_\R\C(S)-1,
	\end{equation}
	and $\s\lra_{\wt\phi}\s'$. 
	Moreover, if $\xi^1,\ldots,\xi^\cp$ 
	for $\cp\coloneqq\dim_\R\C(S)$ is a basis of
	$\C(S)$, and if we let $k$ be such that
	$\xi^{k}_S(\s)\neq\xi^{k}_S(\s')$, then
	\begin{equation}\label{eq: xi_tilde}
		\wt\xi^i\coloneqq\xi^i-\frac{\xi^i_S(\s')-\xi^i_S(\s)}{\xi^{k}_S(\s')-\xi^{k}_S(\s)}\xi^{k},
		\quad i=1,\ldots,\cp,\quad i\neq k
	\end{equation}
	form a basis of $\Cwt(S)$.
\end{lemma}

\begin{proof}
	For $\vp\coloneqq(\s,\s')$, the set $\wt\phi\coloneqq\phi\cup\{\vp,\bar\vp\}\subset(S\times S)\times(S\times S)$ is an interaction on 
	$S$ containing $\phi$.   For a basis $\xi^1,\ldots,\xi^\cp$
	of $\C(S)$, 
	there exists $k \in \{1,\ldots,\cp\}$ such that
	$\xi^k_S(\s)-\xi^k_S(\s')\neq 0$.
	We define $\wt\xi^i$ for $i\neq k$ as in \cref{eq: xi_tilde}.
	Then by construction, 
	we have $\wt\xi^i\in\C(S)$, and
	$\wt\xi^i_S(\s)=\wt\xi^i_S(\s')$ for any $i\neq k$.
	This shows that $\wt\xi^i$ are conserved quantities for the 
	interaction $\wt\phi=\phi\cup\{\vp,\bar\vp\}$.
	Moreover, since $\xi^i$ for $i=1,\ldots,\cp$ are linearly independent in $\C(S)$,
	this implies that 
	$\wt\xi^i$ for 
	$i\neq k$ are also linearly independent in $\C(S)$.
	Since $\Cwt(S)\subsetneq\C(S)$ and 
	$\wt\xi^i$ for $i\neq k$ are linearly independent,
	we see that $\wt\xi^i$ for $i=1,\ldots,\cp$, 
	$i\neq k$
	form a basis of $\Cwt(S)$.  This shows \cref{eq: dim} as desired.	
\end{proof}

We next introduce the notion of separability and weak equivalence of interactions.
Weak equivalence will be used to compare interactions
on local states with different cardinalities.

\begin{definition}
	We say that an interaction $(S,\phi)$ is \emph{separable}, if for any $s,s'\in S$ such that $s\neq s'$, 
	there exists a conserved
	quantity $\xi\in\C(S)$ such that $\xi(s)\neq\xi(s')$.
	In other words, for any $s,s'\in S$, if $\xi(s)=\xi(s')$ for any $\xi\in\C(S)$, then we have $s=s'$.
\end{definition}

Note that the multi-species exclusion interactions and the $\kappa$-exclusion interactions are both separable interactions.  
The Glauber interaction is an example of an interaction which is not separable.

\begin{definition}\label{def: we}
	We say that interactions $(S,\phi)$ and $(S',\phi')$ are \emph{weakly equivalent}
	for the map $\iota\colon S\rightarrow S'$,
	if the map $\Map(S',\R)\rightarrow\Map(S,\R)$ given by composition with $\iota$
	induces an $\R$-linear isomorphism $\iota^*\colon\C(S')\cong\C(S)$, and there 
	exists a map $\iota'\colon S'\rightarrow S$ inducing 
	an $\R$-linear isomorphism ${\iota'}^*\colon\C(S)\cong\C(S')$ that is
	the inverse of $\iota^*$.  In this case, we denote $(S,\phi)\sim(S',\phi')$, or simply
	$\phi\sim\phi'$.
\end{definition}

Note that weak equivalence is an equivalence relation among interactions. Moreover if $(S,\phi)$ and $(S',\phi')$ are equivalent, then they are weakly equivalent.

For the multi-species exclusion interaction $\phi^\kappa_\ms$ for $S=\{0,\ldots,\kappa\}$
of \cref{def: MS}, the space of conserved quantities satisfies
$\dim_\R\Consv^{\phi^\kappa_\ms}(S)=\kappa$ and is spanned by the standard basis
$\xi^1,\ldots,\xi^\kappa$ such that
$\xi^i(0)=0$ and $\xi^i(j)=\delta_{ij}$ for $i,j=1,\ldots,\kappa$.
Distinct diagonal vertices in 
$S\times S$ are in distinct connected components of the graph $(S\times S,\phi^\kappa_\ms)$.

\begin{lemma}\label{lem: diagonal}
	Let $\phi^\kappa_\ms$ be the multi-species exclusion interaction for $S=\{0,\ldots,\kappa\}$.
	Let $j,k\in S$, $j,k\neq0$ such that $j\neq k$.
	Let $\phi$ be an interaction
	containing $\phi^\kappa_\ms$
	such that $(j,j)\lra_{\phi}(k,k)$ and $\C(S)=\kappa-1$.
	Note such interaction is unique up to equivalence by \cref{prop: less}.
	Then 
	$\xi^i$ for $i\neq j,k$ and
	\[
		\xi^{j}+\xi^k
	\]
	give an $\R$-linear basis of $\C(S)$.
	Moreover, the interaction $\phi$ is not separable.
\end{lemma}

\begin{proof}
	We have $\xi^i_S(j,j)=\xi^i_S(k,k)=0$ for $i\neq j,k$,
	and we have $\xi^{j}_S(j,j)=\xi^k_S(k,k)=2$
	and $\xi^{j}_S(k,k)=\xi^{k}_S(j,j)=0$.
	This shows in particular that $\xi^k_S(j,j)\neq\xi^k_S(k,k)$.
	Our result now follows from \cref{prop: less},
	noting that \cref{eq: xi_tilde} gives $\wt\xi^i=\xi^i$ for $i\neq j,k$ and 
	$\wt\xi^j=\xi^j+\xi^k$ in our case.
	Since $\wt\xi(j)=\wt\xi(k)$ for any $\wt\xi\in\C(S)$, we see that $\phi$ is not separable.
\end{proof}

We next show that the interaction $(S_\kappa,\wt\phi)$ constructed in \cref{lem: diagonal}
for $S_\kappa=\{0,1,\ldots,\kappa\}$
is weakly equivalent to the multi-species exclusion interaction on $S_{\kappa-1}=\{0,\ldots,\kappa-1\}$.

\begin{proposition}\label{prop: diagonal}
	Let $(S_\kappa,\phi^\kappa_\ms)$ be the multi-species exclusion interaction on 
	$S_\kappa=\{0,\ldots,\kappa\}$.
	Let $j,k\in S_\kappa$ such that $j\neq k$,
	and let $\wt\phi$ be the interaction constructed in \cref{lem: diagonal}
	containing $\phi^\kappa_{\ms}$
	such that $(j,j)\lra_{\tilde\phi}(k,k)$ and $\Cwt(S)=\kappa-1$.
	Then $(S_\kappa,\wt\phi)$ is weakly equivalent
	to the multi-species exclusion interaction $(S_{\kappa-1},\phi^{\kappa-1}_\ms)$ 
	via the injection $\iota\colon S_{\kappa-1}\rightarrow S_\kappa$ given by 
	$\iota(i)=i$ for $i<k$ and $\iota(i)=i+1$ for $i\geq k$.
\end{proposition}

\begin{proof}
	Let $(S_\kappa,\wt\phi)$ be the interaction constructed in \cref{prop: less}.	
	By reordering $S_\kappa$ if necessary, we can assume that $j=\kappa-1$ and $k=\kappa$.
	By \cref{lem: diagonal}, $\Cwt(S)$ is spanned by $\wt\xi^i=\xi^i$ for $i < \kappa-1$ and 
	$\wt\xi^{\kappa-1}=\xi^{\kappa-1}+\xi^\kappa$.
	The natural inclusion $S_{\kappa-1}\hookrightarrow S_\kappa$ induces 
	the restriction map $\Map(S_\kappa,\R)\rightarrow\Map(S_{\kappa-1},\R)$,
	which maps the basis $\wt\xi^1,\ldots,\wt\xi^{\kappa-1}$
	of $\Cwt(S)$ to the standard basis $\xi^1,\ldots,\xi^{\kappa-1}$ of $\Consv^{\phi^{\kappa-1}_\ms}(S_{\kappa-1})$. 
	If we let $S_\kappa\rightarrow S_{\kappa-1}$ be the surjective map given by mapping $i<\kappa$ to $i$
	 and $\kappa$ to $\kappa-1$, then this induces an isomorphism between 
	 $\Consv^{\phi^{\kappa-1}_\ms}(S_{\kappa-1})$
	 and $\Cwt(S_\kappa)$.
	Hence by \cref{def: we}, we see that the inclusion 
	$S_{\kappa-1}\hookrightarrow S_\kappa$ gives a weak equivalence 
	 between $(S_{\kappa-1},\phi^{\kappa-1}_\ms)$ and $(S_\kappa,\wt\phi)$ as desired.
\end{proof}

\begin{corollary}\label{cor: diagonal}
	Let $(S_\kappa,\phi^\kappa_\ms)$ be the multi-species exclusion interaction on $S_\kappa=\{0,\ldots,\kappa\}$.
	Let $j,k,l\in S_\kappa$ such that $j\neq k$.	
	Then the interaction $(S_\kappa,\phi')$ constructed in \cref{prop: less}
	such that $(j,l)\lra_{\phi'}(k,l)$
	is equivalent to the interaction $(S_\kappa,\wt\phi)$ constructed in \cref{prop: diagonal}
	such that $(j,j)\lra_{\wt\phi}(k,k)$.
	In particular,  the interaction $(S_\kappa,\phi')$ is again weakly equivalent
	to the multi-species exclusion interaction 
	$(S_{\kappa-1},\phi^{\kappa-1}_\ms)$ for $S_{\kappa-1}=\{0,\ldots,\kappa-1\}$.
\end{corollary}

\begin{figure}[htbp]
	\begin{center}
		\includegraphics[width=1\linewidth]{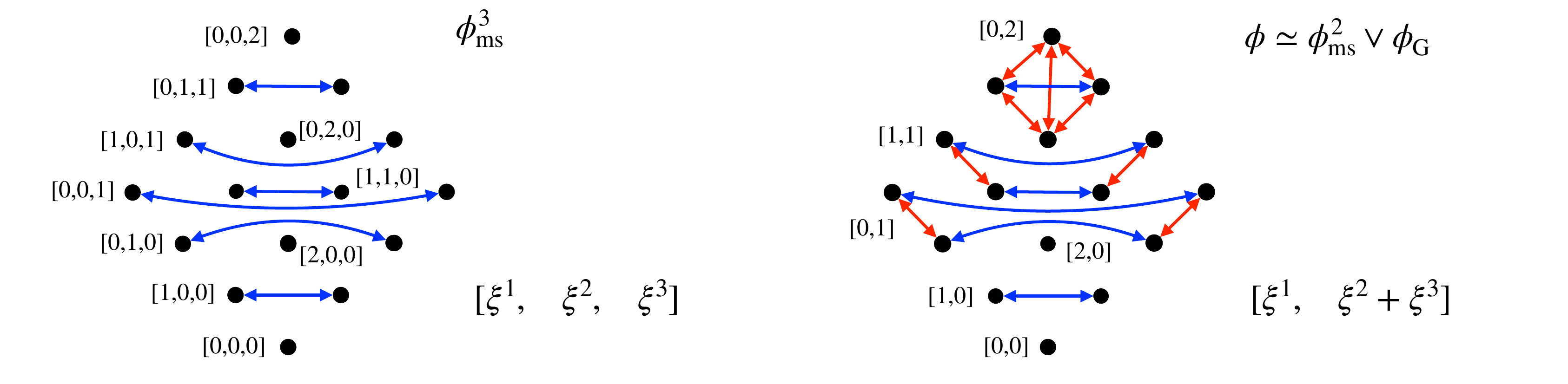}
		\caption{For the multi-species 
		exclusion interaction $\phi^3_\ms$, \cref{prop: diagonal} and \cref{cor: diagonal} 
		shows that any interaction connecting one or more of the red edges 
		on the graph on the right are all equivalent.
		The numbers next to each connected component of the graphs are the values of the conserved quantities.}
		\label{Fig: 11}
	\end{center}
\end{figure}

\begin{proof}
	By reordering the elements of $S=S_\kappa$ and taking equivalent interactions if necessary,
	we can assume that $j=\kappa-1$ and $k=\kappa$.  
	By \cref{lem: diagonal},
	$\Cwt(S)$ is spanned by the conserved quantities 
	$\xi^i$ for $i=1,\ldots,\kappa-2$ and
	$
		\wt\xi^{\kappa-1}\coloneqq\xi^{\kappa-1}+\xi^\kappa,
	$
	where $\xi^1,\ldots,\xi^\kappa$ is the standard basis of $\Consv^{\phi^\kappa_\ms}(S)$.
	We have
	\[
		\xi^i_S(\kappa-1,l)=\xi^i_S(\kappa,l),\quad i=1,\ldots,\kappa-2
	\]
	and
	\[
		\wt\xi^{\kappa-1}_S(\kappa-1,l)=\wt\xi^{\kappa-1}_S(\kappa,l).
	\]
	This implies in particular that $\Cwt(S)\subset\Cd(S)$.
	Since the dimensions of $\Cwt(S)$ and $\Cd(S)$ are both $\kappa-1$,
	we see that  $\Cwt(S)=\Cd(S)$,
	hence
	the interactions $(S,\wt\phi)$ and $(S,\phi')$ are equivalent as desired.
	The last statement of our assertion is a consequence of \cref{prop: diagonal}.
\end{proof}

\begin{remark}\label{rem: collapse}
	\cref{prop: diagonal} and \cref{cor: diagonal} shows that for $(S,\phi^\kappa_\ms)$
	and distinct $a,b\in S$,
	the interaction obtained by connecting diagonal vertices $(a,a)\lra(b,b)$
	and connecting the vertices $(a,c)\lra(b,c)$
	for some $c\in S$ are all equivalent and \emph{not} separable.
\end{remark}

In what follows,
let $\kappa\geq 0$ be an integer and let $S=\{0,\ldots,\kappa\}$.
By \cref{prop: smallest}, we see that any interaction $(S,\phi)$ such that $\cp=\kappa$
is equivalent to the multi-species interaction $(S,\phi^\kappa_\ms)$.
Using \cref{prop: diagonal} and \cref{cor: diagonal}, we can classify equivalence classes
of interactions on $S$ satisfying $\cp=\kappa-1$ for $\kappa\geq 1$ as follows.

\begin{theorem}\label{thm: 1}
	Let $\kappa\geq 1$ be an integer and let $S=\{0,\ldots,\kappa\}$.
	For any interaction $(S,\phi)$, let $c_\phi=\dim_\R\C(S)$.
	Then the equivalence class of interactions on $S$ such that $c_\phi=\kappa-1$ can be 
	classified as follows:
	\begin{enumerate}
		\item If $\kappa=1$, then there exists a \emph{unique} equivalence class of interaction $(S,\phi)$ such that $c_\phi=\kappa-1=0$, which contains the Glauber interaction.
		\item If $\kappa=2$, then there exists exactly \emph{two} equivalence classes of interactions $(S,\phi)$ such that $c_\phi=\kappa-1=1$. 
		One is the class of interactions that is weakly equivalent to $(\{0,1\},\phi^1_\ms)$
		and the other is the equivalence class of the $\kappa$-exclusion interaction 
		$(\{0,1,2\},\phitwEX)$ (see \cref{def: GE}).
		\item If $\kappa\geq 3$, then there exists exactly \emph{three} equivalence classes of interactions $(S,\phi)$ such that $c_\phi=\kappa-1$. 
		For the particular case that $\kappa=3$ so that $S=\{0,1,2,3\}$, the first
		 is the class of interactions that is weakly equivalent to $(\{0,1,2\},\phi^2_\ms)$,
		the second is the equivalence class of the
		 lattice gas with energy $\phitLGE$ on $S$
		 (see \cref{def: LGE}),		
		and the third is the equivalence class of the two-lane particle
		 interaction $\phi^{\square 2}_{\EX}$ on $S$ with the identification $S=\{0,1\}\times\{0,1\}$
		 given by $0\mapsto(0,0)$, $1\mapsto (1,0)$, $2\mapsto (0,1)$ and $3\mapsto(1,1)$
		  (see \cref{def: RP}).
	\end{enumerate}
\end{theorem}

\begin{proof}
	By \cref{prop: smallest}, any equivalence class of interactions
	contains an interaction $(S,\phi)$ such that $\phi^\kappa_\ms\subset\phi$
	for the multi-species exclusion interaction $(S,\phi^\kappa_\ms)$.
	Since $\dim_\R\Consv^{\phi^\kappa_\ms}(S)=\kappa$, 
	if $\dim_\R\C(S)=\kappa-1$, then $(S,\phi)$ is obtained from $(S,\phi^\kappa_\ms)$ by
	connecting two vertices in $(S\times S,\phi^\kappa_\ms)$.
	By reordering the elements of $S$, the connection
	of vertices can be classified into the following cases:
	 (i) the connection of two diagonals $(a,a)\lra(b,b)$ for distinct $a,b\in S$,
	(i)${}'$ the connection of vertices with common component 
	$(a,c)\lra(b,c)$ for distinct $a,b\in S$ and $c\in S$,
	(ii) the connection of a diagonal and a non-diagonal without common components
	$(a,a)\lra(b,c)$ for distinct $a,b,c\in S$, and (iii) 	
	the connection of two non-diagonal vertices without common components
	$(a,b)\lra(c,d)$ for distinct $a,b,c,d\in S$.
	By \cref{rem: collapse}, cases (i) and (i)${}'$ give equivalent interactions.
	Hence there are at most three cases (i)(ii)(iii).
	
	(1) If $\kappa=1$ so that $S=\{0,1\}$, 
	then since $S$ contains only two distinct elements, 
	only case (i) is possible. This gives the unique equivalence class up to isomorphism
	such that $\cp=\kappa-1=0$, which includes the Glauber interaction.

	(2) If $\kappa=2$ so that $S=\{0,1,2\}$,
	then since $S$ has three distinct elements,
	the equivalence class of interactions such that $\cp=\kappa-1=1$
	corresponds to cases (i) or (ii).
	(i) gives an interaction equivalent to $(\{0,1\},\phi^1_\ms)$
	and (ii) gives an interaction equivalent to the $2$-exclusion $\phitwEX$,
	obtained by $(1,1)\lra(0,2)$.
	
	(3) If $\kappa\geq 3$, then we have all three cases (i)(ii)(iii).
	Note that
	(i) gives an interaction weakly equivalent to $(\hat{S},\phi^{\kappa-1}_\ms)$
	for $\hat{S}=\{0,\ldots,\kappa-1\}$.
	In the special case that $\kappa=3$, then by connecting $(2,2)$ with $(1,3)$,
	we see that this case is the equivalence class containing the lattice gas with energy
	given in \cref{def: LGE} (see also \cref{Fig: 10}).
	Furthermore, by connecting $(0,3)$ with $(1,2)$,
	we see that this case is the
	equivalence class containg the two-lane particle interaction $\phi^{\square 2}_\EX=\phiEX\square\phiEX$. 
	This completes the proof of our assertion.
\end{proof}

We can now classify the equivalence classes of interactions on $S=\{0,1\}$, $S=\{0,1,2\}$ and 
$S=\{0,1,2,3\}$. 
 The following is an immediate consequence of \cref{thm: 1}.

\begin{corollary}\label{cor: 1}
	Let $S=\{0,1\}$.
	Then there are exactly 2 equivalence classes of interactions on $S$,
	which are equivalent to the following interactions:
	\begin{enumerate}
		\item The  exclusion interaction given in \cref{def: EI}.
		We have $\cp=1$, and $\C(S)$ is spanned by $\xi$ satisfying $\xi(s)=s$ for $s=0,1\in S$.
		\item The complete graph on $S\times S$,
		which satisfies $\cp=0$ since there is exactly one connected component.
		This is equivalent to the Glauber interaction \cref{example: GI}, and is
		weakly equivalent to the unique interaction with one element
		in the set of local states.
	\end{enumerate}
	Note that (1) is separable whereas (2) is not.
\end{corollary}

\begin{proof}
	The classification follows from \cref{prop: smallest} and
	\cref{thm: 1}, noting that the multi-species interaction for $\kappa=1$
	coincides with the exclusion interaction (see \cref{def: MS}), and
	any interaction $\phi$
	such that $c_\phi=0$, the graph $(S\times S,\phi)$ is equivalent to the complete graph.
\end{proof}

For $S=\{0,1,2\}$, we have the following.

\begin{corollary}\label{cor: 2}
	Let $S=\{0,1,2\}$.
	Then there are exactly 4 equivalence classes of interactions on $S$,
	which are equivalent to the following interactions:
	\begin{enumerate}
		\item The multi-species exclusion interaction defined in \cref{def: MS}.
		We have $\cp=2$, and $\C(S)$ is spanned by the standard basis $\xi^1,\xi^2$.
		\item The interaction weakly equivalent to the multi-species exclusion interactions for $S=\{0,1\}$.  We have $\cp=1$, and $\C(S)$ is spanned by $\xi^1+\xi^2$.
		\item The $\kappa$-exclusion interaction (see \cref{def: GE}).  We have $\cp=1$,
		and $\C(S)$ is spanned by $\xi^1+2\xi^2$. 
		\item The complete graph on $S\times S$,
		which satisfies $\cp=0$ since there is exactly one connected component.
		This is equivalent to the generalized Glauber interaction (see \cref{Fig: 09}), and 
		is weakly equivalent to the unique interaction with one element
		in the set of local states.
	\end{enumerate}
	Note that (1) and (3) are separable, whereas (2) and (4) are not.
\end{corollary}

\begin{proof}
	The classification follows from \cref{prop: smallest} and
	\cref{thm: 1}, noting that any interaction $\phi$
	such that $c_\phi=0$, the graph $(S\times S,\phi)$ is equivalent to the complete graph.
	The calculation of conserved quantities follows from \cref{prop: less}, noting that
	the interaction
	(2) is obtained by connecting $(1,1)$ to $(2,2)$ and
	interaction (3) is obtained by connecting 
	$(1,1)$ with $(0,2)$ and $(2,0)$.
\end{proof}

The equivalence classes of interactions for the case $S=\{0,1,2\}$ given in \cref{cor: 2} is represented by
the interactions in \cref{Fig: 09}.

\begin{figure}[htbp]
	\begin{center}
		\includegraphics[width=0.8\linewidth]{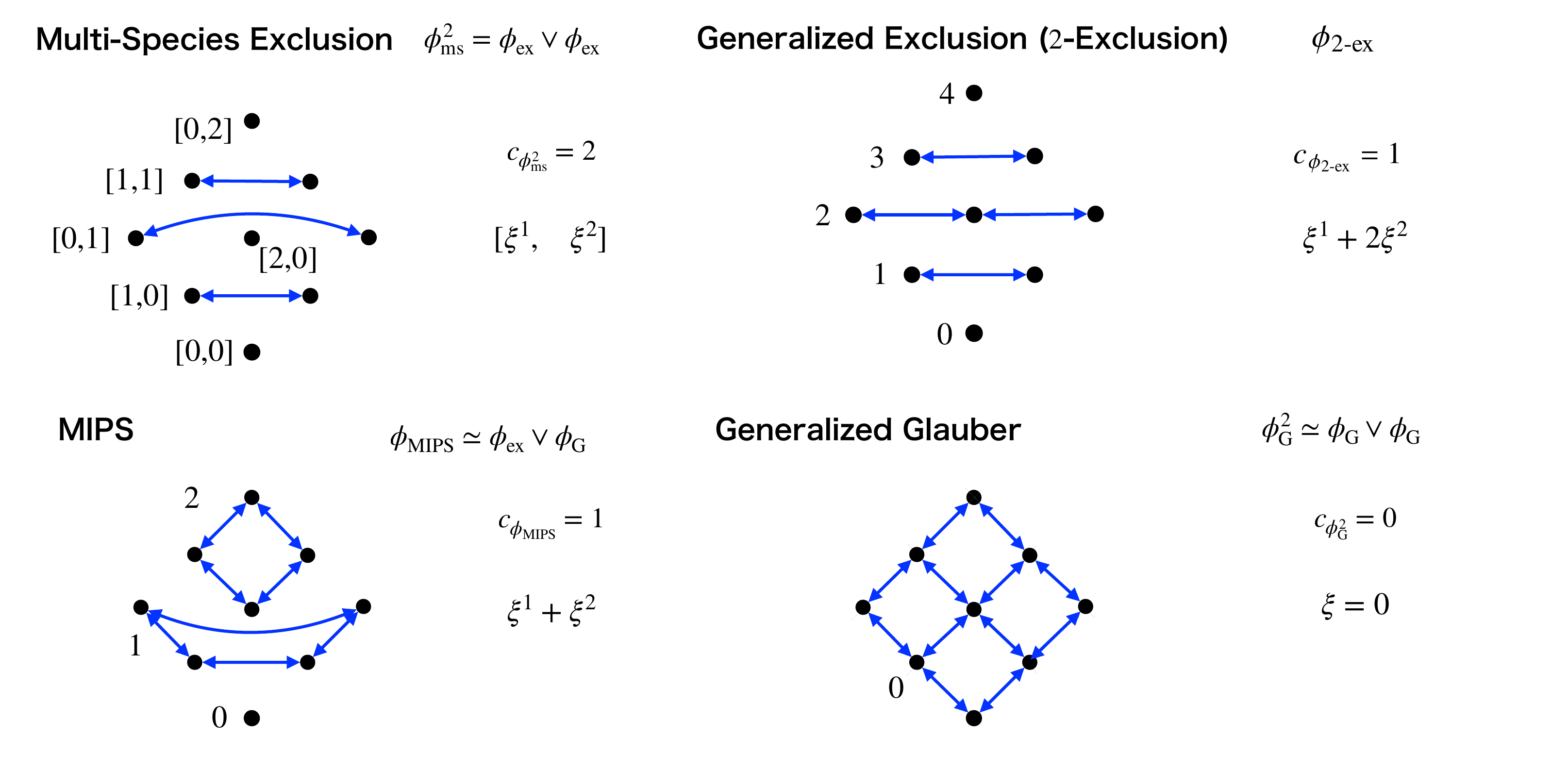}
		\caption{The Classifications of Interactions for $S=\{0,1,2\}$.
		The numbers next to each connected components of the 
		graphs are the values of the conserved quantities.
		The Generalized Glauber Interaction $\phi^2_{\operatorname{G}}$ 
		given via the graph in the above figure can be viewed as the generalization 
		of the Glauber Interaction with spin $\{-1,0,1\}$.}\label{Fig: 09}
	\end{center}
\end{figure}

We next consider the case $S=\{0,1,2,3\}$.

\begin{corollary}
	Let  $S=\{0,1,2,3\}$.  Any interaction on $S$ which is not separable is weakly equivalent to an interaction
	on $\{0,1,2\}$.  Hence any interaction on $S$ which is not
	separable corresponds via weak equivalent to one of the interactions of \cref{cor: 2}.
\end{corollary}

\begin{proof}
	Suppose the interaction $(S,\phi)$ is \emph{not} separable.  
	By definition, there exists $j,k\in S$ such that $j\neq k$  and $\xi(j)=\xi(k)$ for any $\xi\in\C(S)$.
	By rearranging the local states, we can assume that $j=2$ and $k=3$.
	Then for any $l\in S$, we have $\xi_S(2,l)=\xi_S(3,l)$ for any $\xi\in\C(S)$.
	Consider the map $\iota\colon S\rightarrow S'\coloneqq\{0,1,2\}$ by $j\mapsto j$ for $j=0,1,2$ and $3\mapsto 2$
	and let $\phi'\subset (S'\times S')\times (S'\times S')$ be the image of $\phi$ with respect to this map.
	Then map $\iota$ and the inclusion $\iota'\colon S'\hookrightarrow S$ gives a weak equivalence
	$(S,\phi)\sim(S',\phi')$.   Hence $(S,\phi)$ is weakly equivalent to an interaction on $S'$.
	On the other hand, given an interaction $(S',\phi')$, let $\phi\subset (S\times S)\times (S\times S)$ be the pre-image of $\phi'$ via $\iota$.
	 Then we obtain an interaction $\phi$ on $S$ such that $\xi_S(2)=\xi_S(3)$ for any $\xi\in\C(S)$.
	 This proves our assertion.
\end{proof}

For separable interactions on $S=\{0,1,2,3\}$, we have the following.

\begin{figure}[htbp]
	\begin{center}
		\includegraphics[width=1\linewidth]{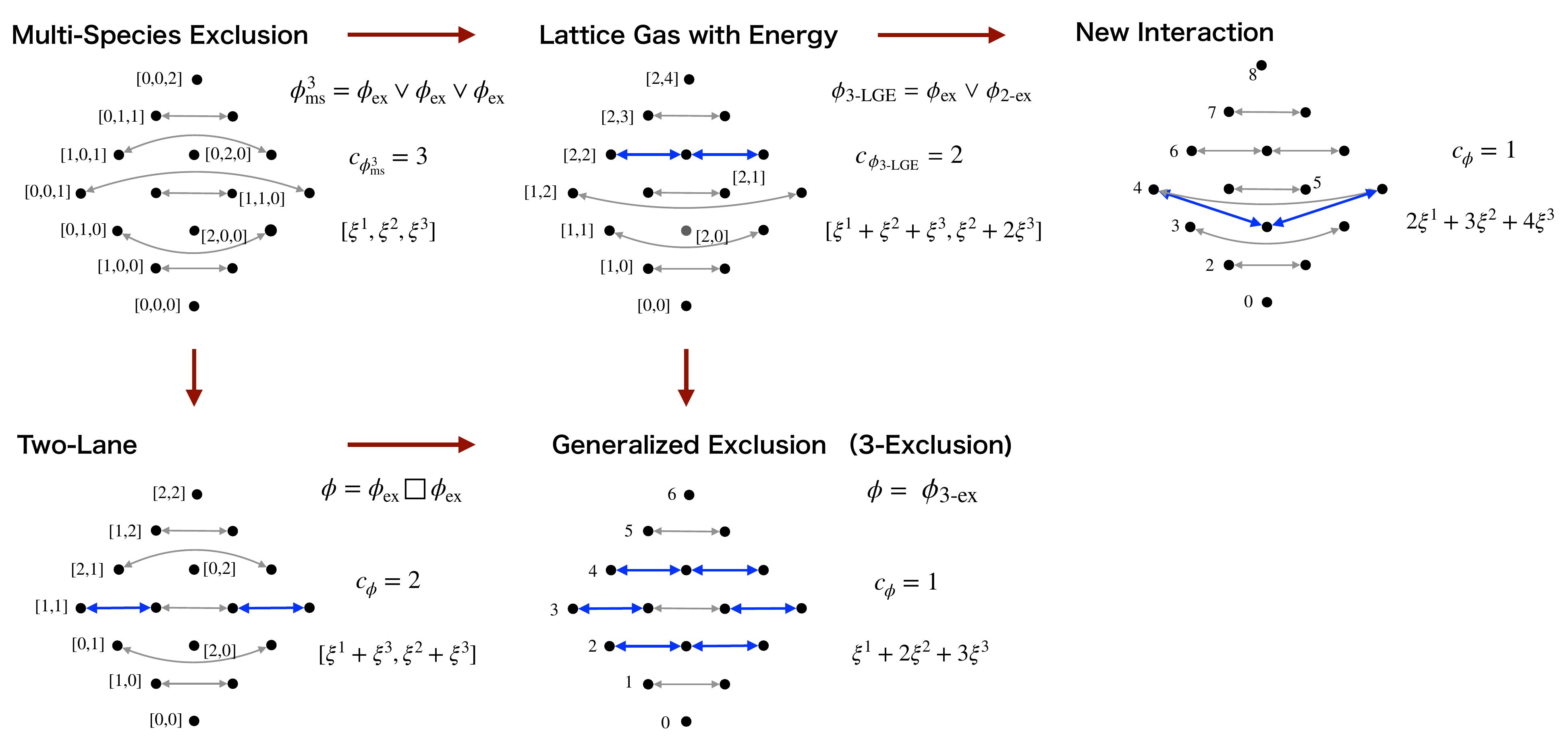}
		\caption{Classification of Equivalence Classes of Separable Interactions for $S=\{0,1,2,3\}$ with $\cp\geq 1$.
		The numbers next to the connected components are the values of the conserved quantities of that component.
		The red arrows indicate addition of an extra edge connecting two distinct connected components of the associated
		graph of the interaction.}
		\label{Fig: 10}
	\end{center}
\end{figure}

\begin{corollary}\label{cor: 3}
	Let $S=\{0,1,2,3\}$.  Then there are exactly 5 equivalence classes of 
	separable interactions on $S$, which are equivalent to the following interactions (1)--(5).
	See \cref{Fig: 10}.
	\begin{enumerate}
		\item The multi-species exclusion interaction $(S,\phi^3_\ms)$ defined in \cref{def: MS}.
		We have $\cp=3$, and $\C(S)$ is spanned by the standard basis $\xi^1,\xi^2,\xi^3$.
		\item  The lattice gas with energy $\phitLGE$ on $S$ (see \cref{def: LGE}).  We have $\cp=2$,
		and $\C(S)$ is spanned by $\xi^1+\xi^2+\xi^3$ and $\xi^2+2\xi^3$. 
		\item The two-lane exclusion interaction on $S$ 
		 (see \cref{def: RP},
		 with $N=2$),
		with the identification $S=\{0,1\}\times\{0,1\}$ mapping $0\mapsto (0,0)$,
		$1\mapsto (1,0)$, $2\mapsto (0,1)$ and $3\mapsto (1,1)$.
		We have $\cp=2$, and $\C(S)$ is spanned by $\xi^1+\xi^3$ and $\xi^2+\xi^3$.
		\item The $3$-exclusion interaction $\phitEX$ (see \cref{def: GE}).  
		We have $\cp=1$,
		and $\C(S)$ is spanned by $\xi^1+2\xi^2+3\xi^3$. 
		\item A new interaction, 
		obtained from (2) by connecting the components $(1,1)$ with $(0,3)$.   
		We have $\cp=1$, and $\C(S)$ is spanned by $2\xi^1+3\xi^2+4\xi^3$. 
	\end{enumerate}
\end{corollary}

\begin{proof}
	By \cref{prop: smallest}, any interaction such that $\cp=3$ is equivalent to $(S,\phi^3_\ms)$.
	This gives our classification (1).
	Any interaction on $S$ is equivalent to an interaction 
	obtained by connecting vertices in $(S\times S,\phi^3_\ms)$.
	We saw in the proof of \cref{thm: 1} that there were four patterns (i)(i')(ii)(iii) in connecting
	distinct vertices of $S\times S$ in an interaction.   Since (i)(i') produces interactions which are not separable,
	any separable interaction such that $\cp=\kappa-1=2$ 
	is obtained is obtained from the multi-species interaction via:
	(ii) the connection of a diagonal and a non-diagonal without common components
	$(a,a)\lra(b,c)$ for distinct $a,b,c\in S$, and (iii) 	
	the connection of two non-diagonal vertices without common components
	$(a,b)\lra(c,d)$ for distinct $a,b,c,d\in S$.
	Our classification (2) corresponds to (ii) 
	by connecting the vertices $(2,2)$ with $(1,3)$,
	and classification (3) corresponds to (iii), by connecting the vertices $(0,3)$ with $(1,2)$.

	The interactions such that $\cp=1$ can be constructed by connecting 
	a pair of distinct connected components in (2) or (3).
	We first consider the case starting from (2).   Since $(2,2)\lra(1,3)$,
	by \cref{rem: collapse}, connecting any other
	vertex to this component will give a non-separable interaction.
	In order to obtain a separable interaction, we have to consider patterns (ii) and (iii)
	connecting vertices outside of $(2,2)$ and $(1,3)$.
	For (ii) the connection of a diagonal and a non-diagonal without common components,
	we first consider the case when the diagonal is $(1,1)$.
	Then $(1,1)$ can connect to $(0,2)$, $(0,3)$ or $(2,3)$.
	If we connect $(1,1)\lra(0,2)$, then by \cref{prop: less},
	this gives an interaction
	whose space of conserved quantities is spanned by $\xi^1+2\xi^2+3\xi^3$,
	which is equivalent to the $3$-exclusion $\phitEX$.
	This corresponds to (4) of our classification.
	If we connect $(1,1)\lra(0,3)$, then this gives an interaction
	whose space of conserved quantities is spanned by $2\xi^1+3\xi^2+4\xi^3$,
	which gives the new interaction in (5) of our classification.
	If we connect
	$(1,1)\lra(2,3)$, then this gives an interaction 
	whose space of conserved quantities is spanned by $\wt\xi\coloneqq\xi^1+\xi^2+\xi^3$.
	Since $\wt\xi(1)=\wt\xi(2)=\wt\xi(3)=1$, this interaction is not separable.
	The case when the diagonal is $(3,3)$ is isomorphic to the case when the diagonal is $(1,1)$
	by exchanging the roles of $1$ and $3$.
	We next consider the case when the diagonal is $(0,0)$.
	Since $(0,0)$ must connect to a vertices not in the connected component containing $(1,3)$,
	the unique option up to exchange of $1$ and $3$ is to connect $(0,0)\lra(1,2)$.	
	We see that this interaction is isomorphic via the permutation $\sigma$ of $S$ 
	given by $\sigma(0)=2$, $\sigma(1)=3$, $\sigma(2)=1$ and $\sigma(3)=0$
	to the new interaction of (5).
	For (iii), since we can not connect to $(1,3)$
	the possible choices are $(0,1)\lra(2,3)$ and $(0,3)\lra(1,2)$,
	which are isomorphic if we exchange the roles of $1$ and $3$.
	If we connect $(0,3)\lra(1,2)$, then this  
	gives the an interaction
	whose space of conserved quantities is spanned by $\xi^1+2\xi^2+3\xi^3$,
	which is again equivalent to the $3$-exclusion $\phitEX$.

	We next consider the case starting from (3), which was obtained from (1) by
	connecting two non-diagonal vertices $(0,3)\lra (1,2)$.	
	In order to obtain a separable interaction, we have to consider patterns (ii) and (iii)
	connecting vertices distinct from $(0,3)$ or $(1,2)$.
	If we consider pattern (ii) connecting
	a diagonal and a non-diagonal vertices without common components,
	then we obtain an interaction isomorphic to an interaction
	obtained by starting from (2) and connecting the vertices $(0,3)\lra(1,2)$.
	This is again equivalent to the $3$-exclusion $\phitEX$.
	Finally, we consider pattern (iii) connecting 
	two non-diagonal vertices of (3) without common components.
	The possibilities are $(0,1)\lra(2,3)$ and $(0,2)\lra(1,3)$,
	which are isomorphic if we replace the roles of $1$ and $2$.
	If we let $\phi'$ be the interaction obtained from (3) by connecting $(0,1)\lra (2,3)$, then 
	$\Cd(S)$ is spanned by $\wt\xi\coloneqq\xi^1+\xi^3$.  
	Since $\wt\xi(1)=\wt\xi(3)=1$, this shows that $(S,\phi')$ 
	is an interaction which is not separable.
	
	We have exhausted all possible combinations for $\cp\geq 1$.
	If $\cp=0$, then the interaction is equivalent to a complete graph,
	hence not separable.
	Our assertion gives the complete classification
	of separable interactions on $S=\{0,1,2,3\}$.
\end{proof}

\begin{remark}
	For the case $S=\{0,1,2,3,4\}$, numerical calculation shows that
	there exists $1$ equivalence class of separable interaction such that $\cp=4$,
	$2$ equivalence classes of separable interaction such that $\cp=3$,
	$8$ equivalence classes of separable interaction such that $\cp=2$,
	and $11$ equivalence classes of separable interaction such that $\cp=1$.
	See \cite{W24}*{Appendix} for a list of irreducibly quantified interactions 
	in this case.
\end{remark}

%
%
%
\section{Irreducibly Quantified Interactions}\label{sec: IQ}
%
%
%

Let $(X,E)$ be a symmetric directed graph.  For any interaction $(S,\phi)$, we call
$\e=(\e_x)\in S^X\coloneqq\prod_{x\in X}S$ a \emph{configuration}, and $S^X$ a \emph{configuration space}.

\begin{lemma}
	For a symmetric directed graph $(X,E)$ and an interaction $(S,\phi)$,
	we let
	\[
		\Phi_E\coloneqq\{ (\e,\e')\in S^X\times S^X\mid\exists e\in E,
		((\e_{oe},\e_{te}),(\e'_{oe},\e'_{te})) \in\phi, \,\,\e_x=\e'_x\,\, \forall x\neq oe,te\}.
	\]
	Then the pair $(S^X,\Phi_E)$ is a symmetric directed graph.
	We call  $(S^X,\Phi_E)$ the configuration space with transition structure of 
	the interaction $(S,\phi)$ on the graph $(X,E)$.
\end{lemma}

\begin{proof}
	Since $\Phi_E\subset S^X\times S^X$, the pair $(S^X,\Phi_E)$ is a  graph.
	If $(\e,\e')\in\Phi_E$, then there exists $e\in E$ such that $\e_x=\e'_x$ for any $x\neq oe,te$ and
	$(\e_e,\e'_e)\in\phi$ where $\eta_e=(\e_{oe},\e_{te})$. 
	Since $(S\times S,\phi)$ is a symmetric graph, we see that $(\e'_e,\e_e)\in\phi$.
	Hence we have $(\e',\e)\in\Phi_E$.  This shows that $(S^X,\Phi_E)$ is symmetric as desired.
\end{proof}

Fix an interaction $(S,\phi)$.
For a finite graph $(X,E)$ and conserved quantity $\xi\colon S\rightarrow\R$, let $\xi_X\colon S^X\rightarrow\R$
be the function defined as
\[
	\xi_X(\e)\coloneqq\sum_{x\in X}\xi(\e_x),\qquad \e=(\e_x)\in S^X.
\] 
The sum is well-defined since we have assumed that $X$ is finite.

Consider the finite graph $(X,E)$ given by $X=\{1,2\}$ and $E=\{(1,2),(2,1)\}\subset X\times X$.
Then we have $S^X=S\times S$ and $\Phi_E=\phi\subset (S\times S)\times (S\times S)$.
Hence $(S^X,\Phi_E)$ coincides with the graph $(S\times S,\phi)$.
For any conserved quantity $\xi\colon S\rightarrow\R$, the function $\xi_X$ coincides with the function $\xi_S\colon S\times S\rightarrow\R$
defined in \S \ref{sec: I}.

The following irreducibly quantified condition, which we view as a condition on the interaction,
plays an important role in the theory of \cite{BKS20}.

\begin{definition}\label{def: IQ}
	We say that an interaction $(S,\phi)$ is \emph{irreducibly quantified}, if it satisfies the following condition:
	For any finite \emph{connected} graph $(X,E)$ and associated configuration space with transition structure
	$(S^X,\Phi_E)$, if $\e,\e'\in S^X$ satisfies $\xi_X(\e)=\xi_X(\e')$ for any conserved quantity $\xi\in\C(S)$,
	then $\e,\e'\in S^X$ are in the same connected component of $(S^X,\Phi_E)$.
\end{definition}

For $\e,\e'\in S^X$, then we denote $\e\lrs_{\Phi_E}\e'$ if $\e$ and $\e'$ are in the same connected
component of $(S^X,\Phi_E)$.  We have the following.

\begin{lemma}
	If $(S,\phi)$ is irreducibly quantified, then
	for any finite \emph{connected} graph $(X,E)$ and associated configuration space with transition structure
	$(S^X,\Phi_E)$,  we have $\xi_X(\e)=\xi_X(\e')$ for any conserved quantity $\xi\in\C(S)$
	if and only if $\e\lrs_{\Phi_E}\e'$.
\end{lemma}

Thus, the irreducibly quantified condition insures that the value of the conserved quantities characterizes
the connected components of the configuration space on any finite connected graph.
This is a crucial property for the theory of hydrodynamic limits.
Note that the irreducibly quantified condition is not in general preserved via equivalence of interactions. 
For example the interaction given in \cref{Fig: 14} is
not irreducibly quantified.  However, it is equivalent to the $3$-exclusion,
which is irreducibly quantified (see \cref{thm: main} (1)).

The irreducibly quantified condition implies that the interaction is separable and exchangeable.

\begin{lemma}\label{lem: FI}
	If an interaction $(S,\phi)$ is irreducibly quantified, then $(S,\phi)$ is separable
	and exchangeable.
\end{lemma}

\begin{proof}
	Suppose $(S,\phi)$ is irreducibly quantified.
	If we consider the connected graph $(X,E)$ given by $X=\{*\}$ and $E=\emptyset$,
	then $S^X=S$ and $\Phi_E=\emptyset$.  
	For any $s,s'\in S$, if $\xi(s)=\xi(s')$ for any conserved quantity $\xi\in\C(S)$,
	then $s\lrs_{\emptyset} s'$, in other words, that $s=s'$.
	This shows that $(S,\phi)$ is separable as desired.
	Next, let $(X,E)$ be the connected graph given by $X=\{1,2\}$ and $E=\{(1,2),(2,1)\}$.
	Let $(s_1,s_2)\in S^X=S\times S$.  Then
	for any conserved quantity $\xi\in\C(S)$, we have 
	$\xi_S(s_1,s_2)=\xi(s_1)+\xi(s_2)=\xi(s_2)+\xi(s_1)=\xi_S(s_2,s_1)$.
	Since $(S,\phi)$ is irreducibly quantified, we have
	$(s_1,s_2)\lrs_{\Phi_E}(s_2,s_1)$,
	hence $(s_1,s_2)\lrs_{\phi}(s_2,s_1)$.  This shows that $(S,\phi)$ is exchangeable
	as desired.
\end{proof}

On the other hand, as in \cref{Fig: 14}, there are examples of interactions which
are separable and exchangeable, but not irreducibly quantified.
See Komiya \cite{Ko23} or Giavarini \cite{G23} for other examples of such interactions.

\begin{figure}[htbp]
	\begin{center}
		\includegraphics[width=0.35\linewidth]{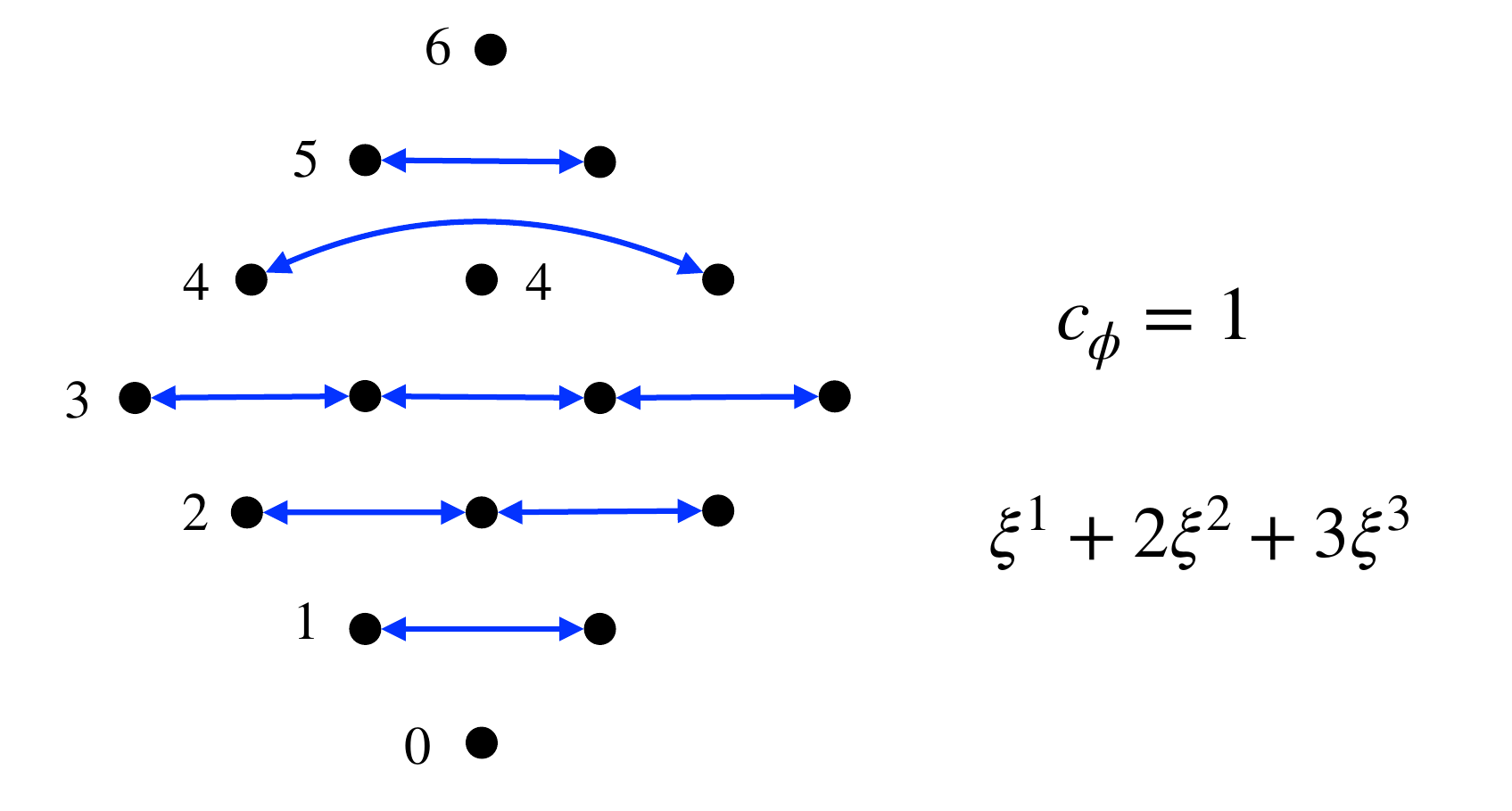}
		\caption{The above interaction is separable and exchangeable, but not irreducibly quantified.
		The numbers next to the vertices expresses the value of the conserved 
		quantity $\xi\coloneqq\xi^1+2\xi^2+3\xi^3$.
		This is equivalent to the $3$-exclusion process $\phitEX$, which is
		irreducibly quantified.}
		\label{Fig: 14}
	\end{center}
\end{figure}

Suppose now that $(S,\phi)$ is an interaction which is exchangeable.
Then by definition, for any $(s_1,s_2)\in S\times S$, we have $(s_1,s_2)\lrs_\phi(s_2,s_1)$.
Let $\e\in S^X$.  For any $x,y\in X$, we define $\e^{x,y}=(\e^{x,y}_z)\in S^X$ 
to be the configuration obtained by exchanging the $x$ and $y$ components of $\e=(\e_z)$.
In other words,
\[
	\e^{x,y}_z\coloneqq
	\begin{cases}
		\e_z &   z\neq x,y,  \\
		\e_y, & z=x, \\ 
		\e_x & x=y.
	\end{cases}
\] 
In particular, for any $e\in E$, we let $\e^{oe,te}$
be the configuration obtained by exchanging the components on the origin and target of $e$.
Note that $\e^{oe,te}_x=\e_x$ for $x\neq oe,te$ and
$(\e^{oe,te}_{oe},\e^{oe,te}_{te})=(\e_{te},\e_{oe})$.
Hence by the definition of $\Phi_E$, we see that $\e\lrs_{\Phi_E}\e^{oe,te}$.

More generally, for an exchangeable interaction, any configuration 
obtained by reshuffling a finite number of components of $\e$
is contained in the same connected component of the configuration space,
as follows:

\begin{lemma}\label{lem: exchangeable}
	Let $(S,\phi)$ be an  interaction which is exchangeable.
	Then for any connected symmetric directed graph $(X,E)$,
	configuration $\e=(\e_z)\in S^X$ and vertices $x,y\in X$,
	we have $\e\lrs_{\Phi_E}\e^{x,y}$.
	Moreover,  if $\e^\sigma=(\e^\sigma_z)\in S^X$ is obtained by shuffling 
	a finite number of components of $\e$ (i.e.\ if there exists a bijection $\sigma: X\rightarrow X$ 
	such that
	$\sigma(z)=z$ for all but finite $z\in X$
	and 
	$\e^\sigma_z\coloneqq\e_{\sigma(z)}$ for any $z\in X$), then we have $\e\lrs_{\Phi_E}\e^\sigma$.
\end{lemma}

\begin{proof}
	If $x=y$, then there is nothing to prove.  Suppose $\abs{X}>1$ and $x\neq y$.
	Since $(X,E)$ is connected, there exists a path $\vec p=(e^1,\ldots,e^N)$ from $x$ to $y$.
	In other words, $x=oe^1$, $y=te^N$, and $te^i=oe^{i+1}$ for $1\leq i<N$.
	We take $\vec p$ to be a minimal path so that $x$ and $te^i$ for $i=1,\ldots,N$ are all distinct.
	Our assertion follows from the fact that $\e^{x,y}$ is obtained from $\e$ as certain compositions of transpositions
	$\e\mapsto\e^{oe^i,te^i}$ along the edges of $\vec p$.
	The second assertion $\e\lrs_{\Phi_E}\e^\sigma$ follows from the fact that any bijection
	$\sigma\colon X\rightarrow X$ obtained by shuffling a finite number of vertices
	can be expressed as a composition of transposition between two elements.
\end{proof}

Using \cref{lem: exchangeable}, we can prove that the wedge sum of irreducibly quantified interactions
is again irreducibly quantified.  We first prove the following.

\begin{lemma}\label{lem: exchange}
	Suppose $(S_1,\phi_1)$ and $(S_2,\phi_2)$ are interactions which are exchangeable.
	Then the wedge sum $(S_1,\phi_1)\vee(S_2,\phi_2)$ along $*_1\in S_1$ and $*_2\in S_2$
	is also exchangeable.
\end{lemma}

\begin{proof}
	Let $(S,\phi)=(S_1,\phi_1)\vee(S_2,\phi_2)$.
	Suppose $(s,s')\in S\times S$.  If $s,s'\in S_1$
	or $s,s'\in S_2$, then $(s,s')\lrs(s',s)$ since $(S_1,\phi_1)$ and $(S_2,\phi_2)$ are
	both exchangeable.  Moreover, if $s\in S_1$ and $s'\in S_2$, then we have
	$(s,s')\lrs_{\phi_1\vee\phi_2}(s',s)$ by the definition of the wedge sum given in \cref{def: WS}.
	This proves our assertion.
\end{proof}

We have the following.

\begin{proposition}\label{prop: WS=IQ}
	Suppose $(S_1,\phi_1)$ and $(S_2,\phi_2)$ are interactions which are irreducibly quantified.
	Then the wedge sum $(S_1,\phi_1)\vee(S_2,\phi_2)$ along $*_1\in S_1$ and $*_2\in S_2$
	is also an irreducibly quantified interaction.
\end{proposition}

\begin{proof}
	Since $(S_1,\phi_1)$ and $(S_2,\phi_2)$ are irreducibly quantified,
	by \cref{lem: FI}, they are both separable.
	Let $(S,\phi)=(S_1,\phi_1)\vee(S_2,\phi_2)$.
	We view $S_1,S_2\subset S$, and $S_1\cap S_2=\{*\}$.
	In order to prove that $(S,\phi)$ is irreducibly quantified,
	consider a finite connected symmetric directed graph $(X,E)$,
	and suppose $\e,\e'\in S^X$ satisfies $\xi_X(\e)=\xi_X(\e')$ for any $\xi\in\C(S)$.
	Since $(S,\phi)$ is exchangeable by \cref{lem: exchange},
	by rearranging the components of $\e$ and $\e'$ as in \cref{lem: exchangeable}
	and exchanging the roles of $\e$ and $\e'$ if necessary,
	we may assume that
	$X_1=\{ x\in X\mid \e_x\in S_1 \}$ and $X'_1= \{ x\in X\mid \e'_x\in S_1\}$
	are both connected and satisfy $X_1'\subset X_1\subset X$.
	Let $\e''=(\e''_z)\in S^X$ be the configuration given by 
	\[
		\e''_z\coloneqq
		\begin{cases}
		 	\e'_z& z\in X'_1, \\
			*& z\in X_1\setminus X'_1,\\
			 \e_z& z\not\in X_1.
		\end{cases}
	\]
	By construction,
 	for any $\xi_1\in\Consv^{\phi_1}(S_1)$, the image of $\xi_1$ in $\C(S)$ satisfies
	\[
		(\xi_1)_{X_1}(\e|_{X_1})=(\xi_1)_{X_1}(\e''|_{X_1}).
	\]
	By applying the irreducibly quantified condition to the configuration space with transition structure
	$(S_1^{X_1},\Phi_{E_1})$  associated to $(S_1,\phi_1)$ on $(X_1,E_1)$
	for $E_1=\{e\in E\mid oe,te\in X_1\}$, we see that the restrictions $\e|_{X_1},\e''|_{X_1}\in S^{X_1}$
	satisfies $\e|_{X_1}\lrs_{\Phi_{E_1}}\e''|_{X_1}$.
	Since the components of $\e$ and $\e''$ coincides outside $X_1$, we see that 
	$\e\lrs_{\Phi_E}\e''$.  
	Next, by construction of $\e''$, by exchanging the components of $\e'$ and $\e''$,
	we may assume that $X_2=\{ x\in X\mid \e'_x\in S_2 \}=\{x\in X\mid \e''_x\in S_2\}\subset X$
	is connected.	
	For any $\xi_2\in\Consv^{\phi_2}(S_2)$, the image in $\C(S)$ satisfies
	\[
		(\xi_2)_{X_2}(\e'|_{X_2})=(\xi_2)_{X_2}(\e''|_{X_2}).
	\]
	By applying the irreducibly quantified condition to 
	$(S_2^{X_2},\Phi_{E_2})$  associated to $(S_2,\phi_2)$ on $(X_2,E_2)$
	for $E_2=\{e\in E\mid oe,te\in X_2\}$, we see that  
	$\e'|_{X_2}\lrs_{\Phi_{E_2}}\e''|_{X_2}$.
	Since the components of $\e'$ and $\e''$ coincides outside $X_2$, we see that 
	$\e'\lrs_{\Phi_E}\e''$.   This proves that 
	$\e\lrs_{\Phi_E}\e'$, hence that
	$(S,\phi)$ is irreducibly quantified
	as desired.
\end{proof}

The box product also preserves the irreducibly quantified condition.

\begin{proposition}\label{prop: BP=IQ}
	Let $(S_1,\phi_1)$ and $(S_2,\phi_2)$ be interactions which are irreducibly quantified.
	Then the box product $(S_1,\phi_1)\square(S_2,\phi_2)$ is also irreducibly quantified.
\end{proposition}

\begin{proof}
	Let $(S,\phi)=(S_1,\phi_1)\square(S_2,\phi_2)$.
	For any finite connected symmetric graph $(X,E)$, consider the configuration
	space with transition structures $(S^X,\Phi_E)$, $(S_1^X,\Phi^1_E)$, $(S_2^X,\Phi^2_E)$ 
	associated to the interactions $(S,\phi)$, $(S_1,\phi_1)$, $(S_2,\phi_2)$.
	Since $S=S_1\times S_2$, we have $S^X=S_1^X\times S_2^X$ and 
	by construction of the transition structure, the graph
	$(S^X,\Phi_E)$ is identified with the box product of the graphs $(S^X_1,\Phi^1_E)$ and $(S^X_2,\Phi^2_E)$,
	given as
	\[
		\Phi_E=\{((\e_1,\e_2),(\e_1',\e_2'))\mid \text{$\e_1=\e'_1$ and $(\e_2,\e'_2)\in \Phi^2_E$
	or $\e_2=\e'_2$ and $(\e_1,\e'_1)\in \Phi^1_E$}\}.
	\]
	For any $\e,\e'\in S^X$, assume that $\xi_X(\e)=\xi_X(\e')$ for any $\xi\in\C(S)$.
	Hence we can write $\e=(\e_1,\e_2)$ and $\e'=(\e'_1,\e'_2)$ with $\e_1,\e'_1\in S_1^X$
	and $\e_2,\e'_2\in S_2^X$.  From our assumption,
	we see that $(\xi_1)_X(\e_1)=(\xi_1)_X(\e'_1)$ for any $\xi_1\in\Consv^{\phi_1}(S_1)$
	and $(\xi_2)_X(\e_2)=(\xi_2)_X(\e'_2)$ for any $\xi_2\in\Consv^{\phi_2}(S_2)$.
	Since $(S_1,\phi_1)$ and $(S_2,\phi_2)$  are both irreducibly quantified,
	we have $\e_1\lrs_{\Phi^1_E}\e'_1$ and
	$\e_2\lrs_{\Phi^2_E}\e'_2$.
	Thus by definition of the box product, we see that $\e''=(\e'_1,\e_2)\lrs_{\Phi_E}\e=(\e_1,\e_2)$
	and $\e''=(\e'_1,\e_2)\lrs_{\Phi_E}\e'=(\e'_1,\e'_2)$.
	This shows that $\e\lrs_{\Phi_E}\e'$ as desired.
\end{proof}

In \cite{BKS20}*{Proposition 2.9}, the irreducibly quantified condition for various interactions
were proved case-by-case by hand.
Using the wedge sum and the box product, we can systematically prove that 
certain interactions are irreducibly quantified.

\begin{theorem}\label{thm: main}
	Let $\kappa\geq 1$ be an integer.
	The following interactions are irreducibly quantified.
	\begin{enumerate}
		\item The $\kappa$-exclusion interaction $\phikEX$ of \cref{def: GE},
		\item The multi-species exclusion interaction $\phi^\kappa_\ms$ of \cref{def: MS}, 
		\item The lattice gas with energy interaction $\phikLGE$ of \cref{def: LGE},
		\item The $N$-lane $\kappa$-exclusion interaction $\phi^{\square N}_{\text{$\kappa$-$\EX$}}$ of \cref{def: RP}.
	\end{enumerate}
\end{theorem}

\begin{proof}
	(1) We first prove the case of the $\kappa$-exclusion interaction.
	The proof is same as that 
	given in \cite{BKS20}*{Proposition 2.19 (2)} for the case of interactions arising from maps.  
	We let $\xi\colon S\rightarrow\R$ be the function given by $\xi(j)=j$ for any $j\in S$.
	Then $\xi$ is a basis of the one-dimensional space $\Consv^{\phikEX}(S)$.
	Let $(X,E)$ be a finite symmetric connected graph.
	Let $\e,\e'\in S^X$ such that $\xi_X(\e)=\xi_X(\e')$.
	Let $\Delta_{\e,\e'}\coloneqq\{x\in X\mid \e_x\neq\e'_x\}$.  If $\abs{\Delta_{\e,\e'}}=0$, 
	then $\e=\e'$ and there is nothing to prove.
	Suppose $\abs{\Delta_{\e,\e'}}\geq 1$, and assume by induction that for any
	configurations $\e''$
	such that $\abs{\Delta_{\e'',\e'}}< \abs{\Delta_{\e,\e'}}$ and $\xi_X(\e')=\xi_X(\e'')$, we have $\e'\lrs_{\Phi_E}\e''$.
	Let $x\in\Delta_{\e,\e'}$.  
	By exchanging the roles of $\e$ and $\e'$ if necessary, we can assume that $\e _x > \e'_x$.
	Since $\xi_X(\e)=\xi_X(\e')$ and $\xi(\e_x)=\e_x\neq\e'_x=\xi(\e'_x)$, there exists $y\in\Delta_{\e,\e'}$
	such that $\e_y<\e'_y$.   
	Fix an edge $e=(oe,te)\in E$.  
	Note that the interaction $\phikEX$ is exchangeable.
	By reshuffling the components of $\e$ and $\e'$
	using \cref{lem: exchangeable}, we can assume that $x=oe$ and $y=te$.
	We denote by $\e''$ the configuration obtained from $\e$ by 
	replacing the $x$ and $y$ components $(\e_x,\e_y)$ with $(\e'_x,\e_y+(\e_x-\e'_x))$.
	Note that for the $\kappa$-exclusion $\phikEX$, we have 
	$(\e_x,\e_y)\lrs_{\phikEX}(\e'_x,\e_y+(\e_x-\e'_x))=(\e''_x,\e''_y)$. 
	This shows that $\e\lrs_{\Phi_E}\e''$.
	Moreover, by construction, $\abs{\Delta_{\e'',\e'}}\leq\abs{\Delta_{\e,\e'}}-1<\abs{\Delta_{\e,\e'}}$.
	Hence by the induction hypothesis, we see that $\e'\lrs_{\Phi_E}\e''$.
	This proves that $\e\lrs_{\Phi_E}\e'$  as desired.
	In particular, the exclusion interaction $\phiEX=\phioEX$ is irreducibly quantified.\\
	(2)
	Since the multi-species exclusion interaction $\phi^\kappa_\ms$ of \cref{def: MS}
	is obtained as the wedge sum $\phi^\kappa_\ms=\vee^{\kappa}\phiEX$ of the
	exclusion interaction, it is irreducibly quantified by \cref{prop: WS=IQ}.\\
	(3)
	Since the lattice gas with energy interaction $\phikLGE$ of \cref{def: LGE}
	is obtained as the wedge sum $\phikLGE=\phiEX\vee\phikmEX$ of the
	exclusion interaction and the $(\kappa-1)$-exclusion interaction, 
	it is irreducibly quantified by \cref{prop: WS=IQ}.\\
	(4)
	Since the $N$-lane $\kappa$-exclusion interaction $\phi^{\square N}_{\text{$\kappa$-$\EX$}}$ of \cref{def: RP}
	is obtained as the box product $\phi^{\square N}_{\text{$\kappa$-$\EX$}}=\phikEX\square\cdots\square\phikEX$ of the
	$\kappa$-exclusion interaction $\phikEX$,
	it is irreducibly quantified by \cref{prop: BP=IQ}.
\end{proof}

We will prove in \cref{appendix: A} that the separable interaction of \cref{cor: 3} (5) is irreducibly quantified.

\appendix
%
%
\section{A New Interaction}\label{appendix: A}
%
%
%

In this appendix, we give two physical interpretations of the New Interaction of \cref{cor: 3} (5), which we discovered from the classification result of interactions for $S=\{0,1,2,3\}$. We will also prove that this new interaction is irreducibly quantified.

The first interpretation is that the new interaction is a variant of the Lattice Gas with Energy
interaction
$\phitLGE$ given in \cref{cor: 3} (2).  See \cref{Fig: 18} for the associated graph with this interpretation.
Since $\phitLGE=\phiEX\vee \phitwEX$, it is a combination of the exclusion
process and the $2$-exclusion process, and can be interpreted as a system with particles
(the exclusion interaction $\phiEX$) with extra energy (the $2$-exclusion interaction
$\phitwEX$).
The extra energy can move between particles on adjacent sites, but can not move to
a vacant site.  See \cref{Fig: 15} for the associated graph for the Lattice Gas with Energy on $S=\{0,1,2,3\}$ with this interpretation. 

\begin{figure}[htbp]
	\begin{center}
		\includegraphics[width=0.8\linewidth]{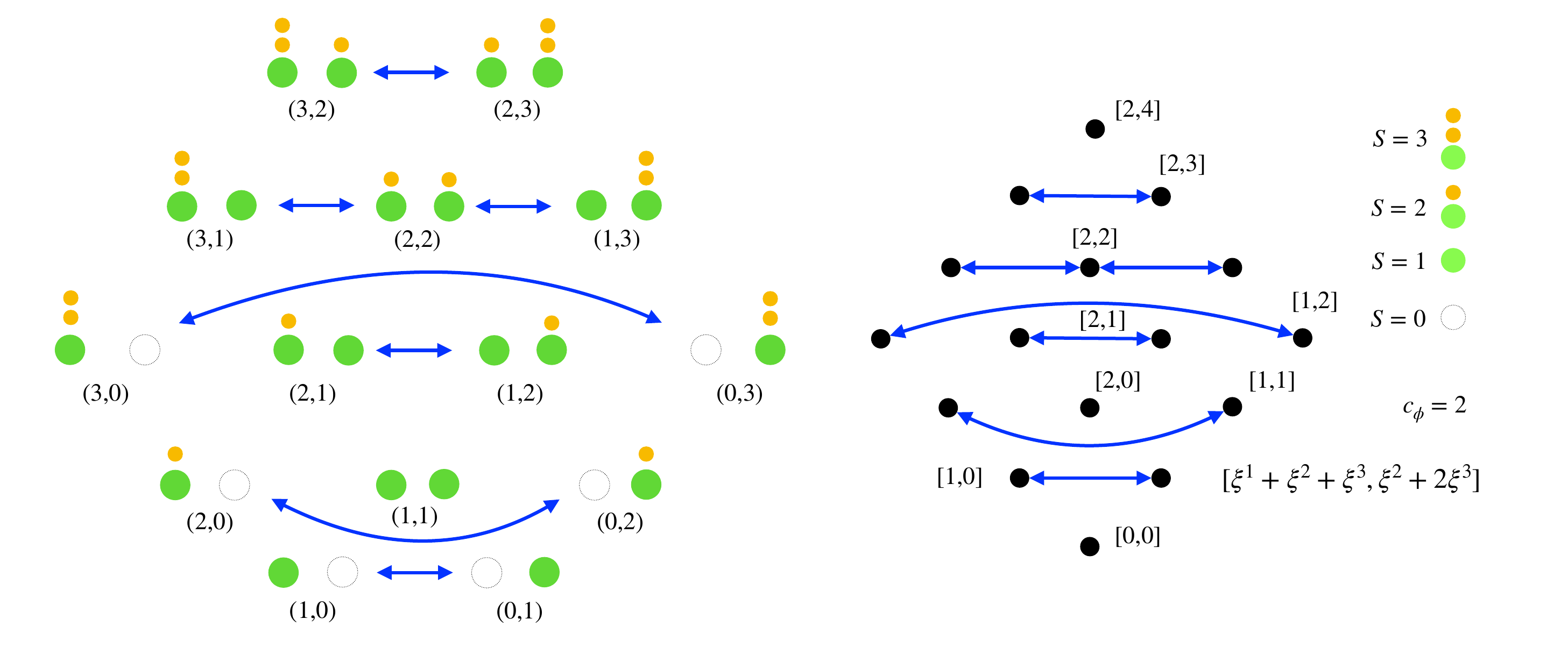}
		\caption{The figure represents the Lattice Gas with Energy on $S=\{0,1,2,3\}$.
		One can interpret the states as either empty ($=0$) or occupied by a single particle
		with no energy ($=1$), a particle with \emph{one} energy ($=2$),
		and a particle with \emph{two} energy ($=3$).
		The conserved quantity $\xi^1+\xi^2+\xi^3$ returns the total number of particles
		and  $\xi^2+2\xi^3$ returns the total energy.}
		\label{Fig: 15}
	\end{center}
\end{figure}

The new interaction of \cref{cor: 3} (5) can be interpreted as the Lattice Gas with Energy
with an additional transformation such that two energy on a single particle adjacent to a vacant site 
can be converted to single particles with no energy.
In other words, two energy is equivalent to a particle with no energy.
The conserved quantity $2\xi^1+3\xi^2+4\xi^3$ gives the total energy of the site,
where a particle is interpreted to have \emph{two} intrinsic energy (see \cref{Fig: 18}), and the total number of particles is no longer conserved. 

\begin{figure}[htbp]
	\begin{center}
		\includegraphics[width=0.8\linewidth]{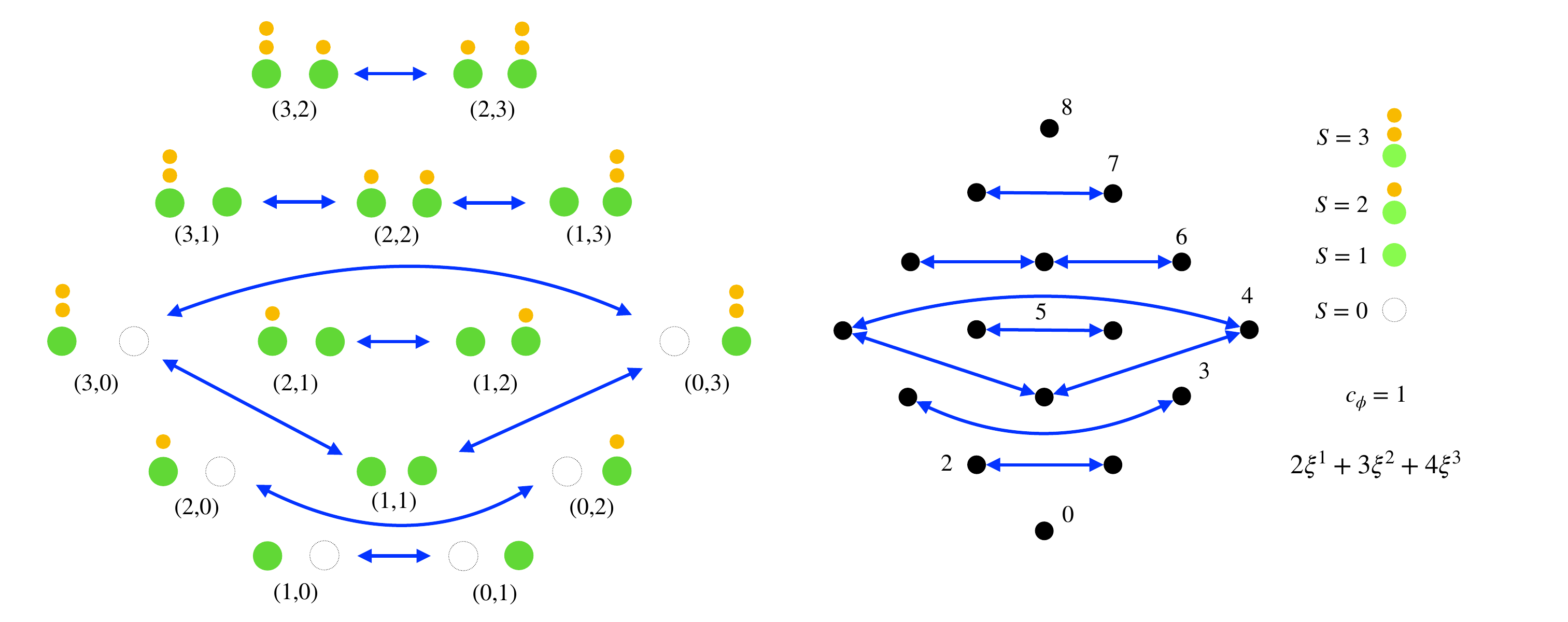}
		\caption{The figure represents the New Interaction.   It can be interpreted as the Lattice Gas with Energy
		with the additional transformation that two extra energy on a single particle adjacent to a vacant site 
		can create a single particle with no extra energy on the vacant site.}\label{Fig: 18}
	\end{center}
\end{figure}

\begin{proposition}
	The new interaction $(S,\phi)$ of \cref{cor: 3} (5) on $S=\{0,1,2,3\}$ is irreducibly quantified.
\end{proposition}

\begin{proof}
	Let $\phi'\coloneqq\phitEX$ be the lattice gas with energy
	so that $\Cd(S)$ is spanned by $\wt\xi^1\coloneqq\xi^1+\xi^2+\xi^3$ and $\wt\xi^2\coloneqq\xi^2+2\xi^3$.
	The new interaction $(S,\phi)$ is obtained from $(S,\phi')$ by connecting $(1,1)$ with $(0,3)$.
	The interaction $(S,\phi)$ is exchangeable.
	By construction, $\phi'\subset\phi$.
	We have $\cp=1$, and $\C(S)$ is spanned by the conserved quantity 
	$\xi\coloneqq2\wt\xi^1+\wt\xi^2=2\xi^1+3\xi^2+4\xi^3$.
	
	For any finite symmetric graph $(X,E)$, let $(S^X,\Phi_E)$
	and $(S^X,\Phi'_E)$
	be the configuration space with
	transition structure associated to the interactions $(S,\phi)$
	and $(S,\phi')$.  For any $\e,\e'\in S^X$, assume that $\xi_X(\e)=\xi_X(\e')$.
	Let $\Delta_1\coloneqq\wt\xi^1_X(\e')-\wt\xi^1_X(\e)$
	and $\Delta_2\coloneqq\wt\xi^2_X(\e)-\wt\xi^2_X(\e')$.
	Then $\Delta_2=2\Delta_1$.  By exchanging the roles of $\e$ and $\e'$ if necessary,
	we assume that $\Delta_1,\Delta_2\geq 0$.
	We prove by induction on $\Delta_1$ that $\e\lrs_{\Phi_E}\e'$.
	If $\Delta_1=0$, then we have $\wt\xi^1_X(\e)=\wt\xi^1_X(\e')$ and $\wt\xi^2_X(\e)=\wt\xi^2_X(\e')$.
	By \cref{thm: main} (3), the interaction
	$(S,\phi')$ is irreducibly quantified.  Hence we have $\e\lrs_{\Phi'_E}\e'$,
	which implies that $\e\lrs_{\Phi_E}\e'$ since $\phi'\subset\phi$.
	Now suppose $\Delta_1>0$, and
	assume that $\hat\e\lrs_{\Phi_E}\e'$ for any $\hat\e\in S^X$
	such that $\xi_X(\hat\e)=\xi_X(\e')$ and $\wt\xi^1_X(\e')-\wt\xi^1_X(\hat\e)<\Delta_1$.
	Since $(S,\phi)$ is exchangeable, we can reorder $\e$ and $\e'$ by \cref{lem: exchange} so that 
	there exists connected subsets $X_1 \subsetneq X'_1\subset X$ such that 
	$\wt\xi^1(\e'_x)=1$ if and only if $x\in X_1'$ and 
	$\wt\xi^1(\e_x)=1$ if and only if $x\in X_1$.
	If $\wt\xi^1_X(\e)=0$, then $\wt\xi^2_X(\e)=0$,
	hence $\xi_X(\e)=\xi_X(\e')=0$, which would imply that $\wt\xi^1_X(\e')=0$.
	This would contradict our assumption that $\Delta_1>0$.
	Hence $\wt\xi^1_X(\e)>0$, so that $X_1\neq\emptyset$.
	Let $e=(oe,te)\in E$ such that $oe\in X_1$ and $te\in X'_1\setminus X_1$.
	Since $\wt\xi^2_X(\e)\geq\Delta_2\geq 2$,
	we can take $\e''\in S^X$ such that $\e\lrs_{\Phi'_E}\e''$,
	$\wt\xi^1(\e''_x)=1$ if and only if $x\in X_1$, and 
	$\e''_{oe}=3$.
	Note $(\e''_{oe},\e''_{te})=(3,0)$.
	  Since $(3,0)\leftrightarrow_{\phi}(1,1)$,
	if we let $\hat\e=(\hat\e_x)\in S^X$ such that $\hat\e_x=\e''_x$ for
	$x\neq oe,te$ and $(\hat\e_{oe},\hat\e_{te})=(1,1)$,
	then $\e''\lra_{\Phi_E}\hat\e$.
	By construction, $\wt\xi^1_X(\e')-\wt\xi^1_X(\hat\e)<\Delta_1$.
	Hence by the induction hypothesis, we have $\hat\e\lrs_{\Phi_E}\e'$.
	This gives $\e\lrs_{\Phi_E}\e'$ as desired.
\end{proof}

\begin{figure}[htbp]
	\begin{center}
		\includegraphics[width=0.8\linewidth]{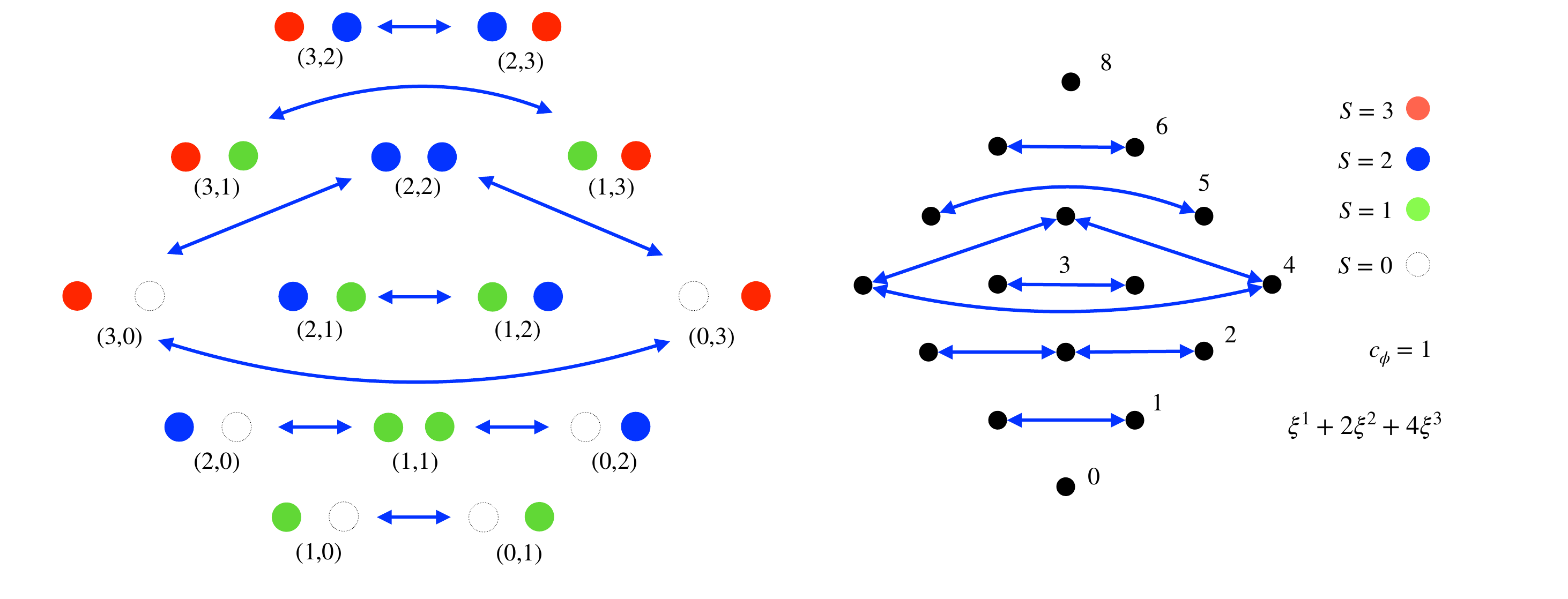}
		\caption{The figure represents an interaction which is isomorphic to the New Interaction
		 of \cref{cor: 3} (5).}
		\label{Fig: 19}
	\end{center}
\end{figure}

The second interpretation of the New Interaction of \cref{cor: 3} (5), obtained by exchanging $1$ and $2$ in
$S=\{0,1,2,3\}$, is as follows. One can interpret the states as either empty ($=0$), 
or occupied by a single green (=1), blue  (=2) and red  (=3) particles.  Two green particles can combine to make a single blue particle, and two blue particles can combine to make a single red particle.
We can interpret the green, blue and red particles having energy $1$, $2$ and $4$, and the conserved
quantity $\xi^1+2\xi^2+4\xi^3$ returns the total energy of the system.

\subsection*{Acknowledgement}
The authors would like to thank Kyosuke Adachi for introducing MIPS to the authors.
One of the motivations for considering the interaction as a graph was to extend the results
of \cite{BKS20} to the MIPS case. The authors thank 
Yusuke Komiya for his Master's thesis project
and Rapha\"el Giavarini for his internship at Keio University, which helped to propel
the classification
of the interactions.  The authors also thank members of the Hydrodynamic Limit Seminar
at Keio/RIKEN, especially Fuyuta Komura, Hayate Suda, Hiroko Sekisaka and Kazuna Kanegae 
for discussion and their continual support for this project.


\end{document}